\documentclass[reqno,centertags,12pt]{amsart}
\usepackage{amsmath,amsthm,amscd,amssymb,latexsym,verbatim}
\usepackage{amssymb}

\usepackage{enumerate}

\usepackage{amsmath}
\usepackage{amsfonts}
\usepackage{amssymb}
\usepackage[english]{babel}
\usepackage{graphicx,epsf,cite,esint}
\usepackage{amsthm}
\usepackage{mathtools}
\usepackage{textcomp}

\usepackage{hyperref}
\usepackage{accents}
\usepackage[T1]{fontenc}
\usepackage[utf8]{inputenc}
\usepackage{empheq}
\def\Xint#1{\mathchoice
{\XXint\displaystyle\textstyle{#1}}%
{\XXint\textstyle\scriptstyle{#1}}%
{\XXint\scriptstyle\scriptscriptstyle{#1}}%
{\XXint\scriptscriptstyle\scriptscriptstyle{#1}}%
\!\int}
\def\XXint#1#2#3{{\setbox0=\hbox{$#1{#2#3}{\int}$ }
\vcenter{\hbox{$#2#3$ }}\kern-.6\wd0}}

\def\dashint{\Xint-}

\usepackage [autostyle, english = american]{csquotes}

\usepackage{graphicx,epsf,cite}
\usepackage{comment}
\usepackage[shortlabels]{enumitem}

\usepackage{xcolor}
\usepackage{mathtools}

-\marginparwidth \oddsidemargin -4mm \evensidemargin \oddsidemargin

\newtheorem{theorem}{Theorem}[section]

\newtheorem{lemma}[theorem]{Lemma}

\theoremstyle{definition}
\newtheorem{definition}[theorem]{Definition}

\newcounter{smalllist}

\allowdisplaybreaks
\numberwithin{equation}{section}

\newcommand{\lb}{\label}

\newcommand{\supp}{\text{\rm{supp}}}

\newcommand{\beq}{\begin{equation}}
\newcommand{\eeq}{\end{equation}}

\newcommand{\bal}{\begin{align}}
\newcommand{\eal}{\end{align}}
\newcommand{\bals}{\begin{align*}}
\newcommand{\eals}{\end{align*}}

\newcommand{\bbN}{{\mathbb{N}}}
\newcommand{\bbR}{{\mathbb{R}}}

\newcommand{\bbS}{{\mathbb{S}}}

\newcommand{\calH}{{\mathcal H}}

\newcommand{\eps}{\varepsilon}

\begin{document}
\title[Propagation of Reactions with Fractional Diffusion]
{Optimal Estimates on the Propagation of Reactions with Fractional Diffusion}

\author{Yuming Paul Zhang and Andrej Zlato\v s}

\address{\noindent Department of Mathematics \& Statistics \\ Auburn University\\ Auburn, AL 36849 \newline Email: \tt
yzhangpaul@auburn.edu}

\address{\noindent Department of Mathematics \\ University of California San Diego \\ La Jolla, CA 92093 \newline Email: \tt
zlatos@ucsd.edu}


\begin{abstract} 
We study the reaction-fractional-diffusion equation $u_t+(-\Delta)^{s} u=f(u)$ with ignition and monostable reactions $f$, and $s\in(0,1)$. We obtain the first optimal bounds on the propagation of front-like solutions in the cases where no traveling fronts exist. Our results cover most of these cases, and also apply to propagation from localized initial data.
\end{abstract}

\maketitle

\section{Introduction} \lb{S1}

In this paper we consider the Cauchy problem for the reaction-fractional-diffusion equation
\beq\lb{1.0}
u_t+(-\Delta)^{s} u=f(u),
\eeq
with $(t,x)\in [0,\infty)\times \bbR^d$ and $f$ a Lipschitz reaction function. 
 One frequently assumes that $f(0)=f(1)=0$ and considers solutions $0\le u\le 1$ that model transitions between two equilibrium states (i.e., $u\equiv 0$ and $u\equiv 1$), driven by the interplay of the two physical processes involved: reaction and diffusion. Our goal is to obtain optimal estimates on the speed of invasion of one equilibrium ($u\equiv 0$) by the other ($u\equiv 1$), so we will study the speeds of propagation of solutions with 
 front-like   (see \eqref{1.11a} below) and localized  initial data. 
Note that the comparison principle shows that in the case of front-like initial data, it suffices to consider \eqref{1.0} in one spatial dimension $d=1$, that is,
\beq\lb{1.1}
u_t+(-\partial_{xx})^{s} u=f(u)
\eeq
 on $[0,\infty)\times \bbR$.  We will do so here when we discuss such initial data, while for localized data we will consider \eqref{1.0} with $d\ge 1$.  The distinction between these two cases is marginal when $s=1$, but this is not anymore the case when $s\in(0,1)$ and diffusion has long range  kernels.

The classical diffusion case $s=1$ goes back to pioneering works by  Kolmogorov, Petrovskii, and Piskunov \cite{KPP}, and Fisher \cite{Fisher}, and it is now well-known that
solutions with both types of initial data propagate ballistically 
for all reaction functions of interest --- monostable, ignition, as well as (unbalanced) bistable 
\cite{aronson1978}. 

We will therefore concentrate here on the fractional diffusion case $s\in (0,1)$, with
 the fractional Laplacian given by
\beq\lb{1.2x}
{
(-\Delta)^s u(x) = c_{s,d} \,\,{\rm p.v.} \int_{\bbR^d} \frac{u(x)-u(y)}{|x-y|^{d+2s}}dy,
}
\eeq
where $c_{s,d}:=c_s \left(\int_{\bbR^{d-1}} (1+h^2)^{-\frac d2 -s} dh \right)^{-1}$ and $c_s=c_{s,1}>0$ is an appropriate constant.
 Then \eqref{1.0} models reactive processes subject to non-local diffusion, mediated by L\' evy stochastic processes with jumps (see, e.g., \cite{34in CGZ} and references therein), and the question of propagation of solutions turns out to be much more complicated.  Its study was initiated by Cabr\' e and Roquejoffre in \cite{cabre2013influence}, who considered \eqref{1.0} with Fisher-KPP reactions (specifically, concave ones with $f'(0)>f(0)=0=f(1)>f'(1)$), which are a special case of 1-monostable reactions from Definition \ref{D.1.1} below.  They proved that solutions to \eqref{1.1} with front-like initial data propagate exponentially, in the sense that $u(t,\cdot)\to 1$ uniformly on $\{x\le e^{\sigma t}\}$ for each $\sigma<\frac{f'(0)}{2s}$, while $u(t,\cdot)\to 0$ uniformly on $\{x\ge e^{\sigma t}\}$ for each $\sigma>\frac{f'(0)}{2s}$. 
 They also considered localized (non-zero non-negative fast-decaying) initial data for \eqref{1.0} and showed that in that case one has $u(t,\cdot)\to 1$ uniformly on $\{|x|\le e^{\sigma t}\}$ for each $\sigma<\frac{f'(0)}{d+2s}$, while $u(t,\cdot)\to 0$ uniformly on $\{|x|\ge e^{\sigma t}\}$ for each $\sigma>\frac{f'(0)}{d+2s}$.
 We  note that prior to \cite{cabre2013influence}, exponential propagation for Fisher-KPP reactions and continuous diffusion kernels with algebraically decreasing tails (from compactly supported initial data in one dimension) was established by Garnier \cite{Gar}.  While there are many other papers studying such questions for various diffusion operators (see, e.g., \cite{BCL2,alfaro2017propagation} and references therein), we will restrict our presentation here to  \eqref{1.0}.

The exponential propagation rates for Fisher-KPP reactions and $s<1$ are due to interaction between the long range kernels of the fractional diffusion and a strong hair-trigger effect of the reaction.  They contrast with the case $s=1$, when level sets of solutions are located in an $o(t)$ neighborhood of the point $x=ct$ (for front-like data) resp.~the sphere $\partial B_{ct}(0)$ (for localized data), with the spreading speed $c$ depending only on $f$ (for all the above types of reactions \cite{aronson1978}).
It turns out that they are in fact a special feature of 1-monostable reactions, and the situation is very different for all the other reaction types.
Let us now define these.


\begin{definition} \lb{D.1.1}
Let $f:[0,1] \to \bbR$ be a Lipschitz continuous function with $f(0)=f(1)=0$. 

\begin{itemize}
\item[(i)] If there is  $\theta_0\in (0,1)$ such that $f(u)=0$ for all $u\in (0,\theta_0]$ and $f(u)>0$ for all $u\in (\theta_0,1)$, then $f$ is an {\it ignition  reaction} and $\theta_0$ is the {\it ignition temperature}.

\item[(ii)] If $f(u)>0$ for all $u\in (0,1)$, then $f$ is a {\it monostable reaction}.  If also
\beq \lb{1.111}
\gamma u^\alpha\le f(u) \le \gamma' u^\alpha \qquad \text{ for all }u\in (0,\theta_0],
\eeq
where $\alpha\ge 1$, $\theta_0\in (0,1)$, and $\gamma,\gamma'>0$,
we say that $f$ is an {\it $\alpha$-monostable reaction}.

\item[(iii)] If there is  $\theta_0\in (0,1)$ such that $f(u)<0$ for all $u\in (0,\theta_0)$ and $f(u)>0$ for all $u\in (\theta_0,1)$, as well as $\int_0^1 f(u) du>0$, then $f$ is an (unbalanced) {\it bistable reaction}.
\end{itemize}
\end{definition}

Ignition reactions are used to model combustive processes, while monostable reactions model phenomena such as chemical kinetics and population dynamics \cite{49xin1992,Fife,gavalas,Berrev}.  Bistable reactions are used in models of phase transitions and nerve pulse propagation  \cite{AC,AR,NAY}, and the unbalanced condition $\int_0^1 f(u) du>0$ guarantees at least ballistic propagation 
for all non-negative solutions that are initially larger than $\theta$ (for any $\theta>\theta_0$) on some large-enough ball (of $\theta$-dependent radius).  This is also the case for  all ignition and monostable reactions (with $\theta>0$ in the latter case).

In particular,  the following holds for all the reactions from Definition \ref{D.1.1}.  If $0\le u\le 1$ is a solution to \eqref{1.1} with either $\liminf_{x\to -\infty} u(0,x)>\theta_0$ for ignition and bistable $f$, or $\liminf_{x\to -\infty} u(0,x)>0$ for monostable $f$, then for any $\lambda\in (0,1)$ we have 
\beq\lb{1.66}
\liminf_{t\to\infty} \frac {{\underline x}_\lambda(t;u)} t>0,
\eeq
 where for  $t\geq 0$ we let
\beq\lb{1.55a}
{\underline x}_\lambda(t;u):=\inf \left\{ x\in\bbR\,\big|\, u(t,x)\leq \lambda \right\}
\eeq
 be the left end of the $\lambda$-level set of $u(t,\cdot)$.  If we instead let
 \beq\lb{1.55a'}
{\underline x}_\lambda(t;u):=\inf \left\{ |x| \,\big|\, u(t,x)\leq \lambda \right\},
\eeq
this claim also extends to solutions to \eqref{1.0} in any dimension and with  $\inf_{|x|\le R_\theta} u(0,x)\ge\theta$, where  $\theta$ must satisfy either $\theta>\theta_0$ for ignition and bistable $f$, or $\theta>0$ for monostable $f$ (and $R_\theta$ also depends on $s,f,d$).  
Both these claims easily follow from the proof of the last claim in Lemma~\ref{L.2.4} below.  (We note that when $f$ is sufficiently small near $u=0$, solutions with small enough initial data may be {\it quenched} in the sense that $\lim_{t\to\infty} \|u(t,\cdot)\|_\infty=0$.)

Ballistic propagation for  $u$  is therefore equivalent to
\[
\limsup_{t\to\infty} \frac {{\overline x}_\lambda(t;u)} t <\infty
\]
for all $\lambda\in (0,1)$, where for front-like data and \eqref{1.1} we let
\beq\lb{1.55b}
{\overline x}_\lambda(t;u):=\sup \left\{ x\in\bbR\,\big|\, u(t,x)\geq \lambda \right\} \qquad(\ge {\underline x}_\lambda(t;u))
\eeq
be the right end of the $\lambda$-level set of $u(t,\cdot)$, and for localized data and \eqref{1.0} we let
 \beq\lb{1.55b'}
{\overline x}_\lambda(t;u):=\sup \left\{ |x| \,\big|\, u(t,x)\geq \lambda \right\} \qquad(\ge {\underline x}_\lambda(t;u)).
\eeq
Comparison principle shows that this holds whenever we have either $u(0,\cdot)\le \theta \chi_{(-\infty,R)}$ or  $u(0,\cdot)\le \theta \chi_{B_R(0)}$ for some $\theta<1$ and $R\in\bbR$, provided \eqref{1.1} has a {\it traveling front}.  The latter is a solution of the form $\tilde u(t,x)=U(x-ct)$, with $\lim_{x\to-\infty} U(x)=1$ and $\lim_{x\to\infty} U(x)=0$ (i.e., $U$ must satisfy $-cU_x+(-\partial_{xx})^{s} U=f(U)$).  In fact, it suffices to have $u(0,\cdot)\le \chi_{(-\infty,R)}$ or  $u(0,\cdot)\le  \chi_{B_R(0)}$ as long as a traveling front exists for \eqref{1.1} with some $\tilde f\ge f$ in place of $f$, such that $\tilde f(1+\delta)=0$ and $\tilde f>0$ on $[1,1+\delta)$ for some $\delta>0$ (then of course $\lim_{x\to-\infty} U(x)=1+\delta$, and one only needs $u(0,\cdot)$ to be dominated by some shift of $U$).  Hence in the rest of this discussion we will assume that
\beq\lb{1.11a}
\theta \chi_{(-\infty,0)}\le  u(0,\cdot) \leq \chi_{(-\infty,R)}
\eeq
for front-like data and
\beq\lb{1.11a'}
\theta \chi_{B_{R'}(0)} \le  u(0,\cdot) \leq \chi_{B_R(0)}
\eeq
for localized data, with some $\theta,R,R'>0$ (and $\theta>\theta_0$ when $f$ is  ignition or bistable).

Proving existence of traveling fronts for \eqref{1.1}, and hence ballistic propagation of solutions, requires one to solve only a (non-local) ODE, and this was indeed achieved in a number of cases.  These include all the above reactions with $s=1$ \cite{aronson1978,AR}, where diffusion is local,  as well as all $C^2$ bistable reactions with any $s\in(0,1)$ \cite{achleitner2015traveling,gui2015,chmaj2013existence}, where  the negative values of $f$ near $u=0$ suppress the effects of long range dispersal.  The cases of ignition and monostable reactions with $s\in(0,1)$ are more delicate, and depend intimately on the interplay of the long range diffusion and the strength of $f$ near $u=0$.  Nevertheless, Mellet, Roquejoffre, and Sire   proved that traveling fronts still exist for ignition reactions with $f'(1)<0$ when $s>\frac{1}{2}$ \cite{mellet2011existence}, while Gui and Huan later showed that they do not exist when $s\le \frac{1}{2}$, as well as that they exist for $\alpha$-monostable reactions (and $s\in(0,1)$) precisely when $s\ge \frac{\alpha}{2(\alpha-1)}$ \cite{gui2015traveling} ($f$ was assumed to satisfy additional hypotheses in \cite{gui2015traveling} when $s>\frac 12$).  We note that since the comparison principle can be used to estimate propagation of solutions for monostable reactions that lie between multiples of  two distinct powers of $u$ near $u=0$, it makes sense to concentrate only on $\alpha$-monostable reactions among the monostable ones; we will do so here.

This leaves one with an expectation of super-ballistic (i.e., accelerating) propagation in the cases of ignition reactions with $s\le \frac 12$ and $\alpha$-monostable reactions with $s< \min\{\frac{\alpha}{2(\alpha-1)},1\}$. For  front-like initial data and \eqref{1.1}, this has indeed been confirmed in all these cases except for ignition reactions with $s=\frac 1{2}$.   
The result for concave Fisher-KPP reactions in \cite{cabre2013influence}  immediately yields exponential propagation for all 1-monostable reactions and $s\in(0,1)$, albeit with the lower and upper exponential rates being $\frac{\gamma}{2s}$ and $\frac{\gamma'}{2s}$, respectively.
%
%
More recently, Coville, Gui, znd Zhao \cite{coville2020propagation}  proved for $\alpha$-monostable reactions with $\alpha>1$ and $s< \min\{\frac{\alpha}{2(\alpha-1)},1\}$
that
 \[
0< \liminf_{t\to\infty}  t^{-\frac{\max\{\alpha-1,1\}}{2s(\alpha-1)}} \, {\overline x}_\lambda(t;u) 
\qquad\text{and}\qquad  
\limsup_{t\to\infty}  t^{-\frac{\alpha}{2s(\alpha-1)}} \, {\overline x}_\lambda(t;u) <\infty
\]
 for all $\lambda\in(0,1)$   (assuming in addition that $f$ is $C^1$ and $f'(1)<0$), which then also yields for ignition reactions with $s\in(0,\frac 12]$ that
 \beq\lb{1.22a}
\limsup_{t\to\infty}  t^{-\frac{1}{2s}-\eps} \, {\overline x}_\lambda(t;u) <\infty
\eeq
 for all $\eps>0$ and $\lambda\in(0,1)$.  We note that while  \cite[Theorem 1.3]{coville2020propagation} may appear to also imply  $\liminf_{t\to\infty}  t^{-\frac{1}{2s}} \, {\overline x}_\lambda(t;u) >0$ for ignition reactions and all $\lambda\in(0,1)$,  Proposition~3.1 in its proof in fact assumes $f$ to be monostable.
Nevertheless, we still have \eqref{1.66} in this case.
 
While these results cover all the cases of interest in which traveling fronts do not exist, in {\it all of them} there is a gap between the powers  of time resp.~exponential rates in the best available lower and upper bounds on the dynamic: an infinitesimal one for concave Fisher-KPP reactions and for ignition reactions with $s=\frac 1{2}$, and a positive one in all the other cases.  
%
 %
In the following first main result of the present paper, we fully close this gap in 
all the latter cases,
 proving that ${\underline x}_\lambda(t;u)$ and ${\overline x}_\lambda(t;u)$ both have the exact power behavior $O(t^{\frac{\alpha}{2s(\alpha-1)}})$ in time for all $\alpha$-monostable reactions with $\alpha>1$, and $O(t^{\frac{1}{2s}})$ for all ignition reactions (so we improve both the lower and upper bounds in the ignition case).   We do so for all the values of $s$ for which traveling fronts do not exist,
except for ignition reactions with $s=\frac 12$, where almost-ballistic propagation (which follows from 
 \eqref{1.66} and \eqref{1.22a}), remains the best result.

\begin{theorem} \lb{T.1.1}
Let $0\le u\le 1$ be a solution to \eqref{1.1} such that \eqref{1.11a} holds for some $\theta,R>0$, and let ${\underline x}_\lambda(t;u)$ and ${\overline x}_\lambda(t;u)$ be from \eqref{1.55a} and \eqref{1.55b}.

\begin{itemize}
\item[(i)] If $f$ is an ignition reaction with ignition temperature $\theta_0\in(0,\theta)$ and $s\in(0,\frac 12)$, then for each $\lambda\in (0,1)$ 
we have
 \[
0< \liminf_{t\to\infty}  t^{-\frac{1}{2s}} \, {\underline x}_\lambda(t;u) \le
\limsup_{t\to\infty}  t^{-\frac{1}{2s}} \, {\overline x}_\lambda(t;u) <\infty
\]
\item[(ii)] If $f$ is an $\alpha$-monostable reaction for some $\alpha>1$ and $s\in(0,\min\{\frac{\alpha}{2(\alpha-1)},1\})$, then for each $\lambda\in (0,1)$ 
we have
\[
0< \liminf_{t\to\infty}  t^{-\frac{\alpha}{2s(\alpha-1)}} \, {\underline x}_\lambda(t;u) \le
\limsup_{t\to\infty}  t^{-\frac{\alpha}{2s(\alpha-1)}} \, {\overline x}_\lambda(t;u) <\infty
\]
\end{itemize}
\end{theorem}

{\it Remarks.}  1.  The leading orders of the  propagation rates in both (i) and (ii) only depend on $s$ and  the qualitative behavior of $f$ near 0, and in (ii) they are independent of $\gamma, \gamma'$ from \eqref{1.111}.  In contrast,  for Fisher-KPP reactions they also depend on $f'(0)$ \cite{cabre2013influence}, and so  their dependence on $f$  for general 1-monostable reactions will  be much more sensitive.
\smallskip

2.  As Theorem \ref{T.3.3} below shows, (i) extends to the case when $u(0,\cdot) \leq \chi_{(-\infty,R)}$ is replaced by $\limsup_{x\to\infty} x^{-2s}u(0,x)<\infty$.  The supersolutions constructed in \cite{coville2020propagation} show that in (ii) we can instead allow $\limsup_{x\to\infty} x^{-\frac{2s}\alpha}u(0,x)<\infty$. 
\smallskip

3.  One-sided bounds in \eqref{1.111} obviously yield one-sided bounds in (ii).
\smallskip

4.  For any $s\in(0,\frac12)$, (i) can be regarded as the $\alpha\to\infty$ limit of (ii).
\smallskip

To the best of our knowledge, these are the first qualitatively optimal propagation results for front-like solutions in situations where no traveling fronts exist.  
\medskip


When it comes to localized initial data and \eqref{1.0}, the corresponding Fisher-KPP result from  \cite{cabre2013influence}  was improved by Coulon and Yangari in \cite{coulon2017exponential}.  They proved for each $\lambda\in(0,1)$ that 
\[
0< \liminf_{t\to\infty}  e^{-\frac{f'(0)}{d+2s} t} \, {\underline x}_\lambda(t;u) \le 
\limsup_{t\to\infty}  e^{-\frac{f'(0)}{d+2s} t} \, {\overline x}_\lambda(t;u) <\infty
\]
for solutions with fast-decaying initial data when  $s\in(0,1)$ and $f$ is any $C^1$ $1$-monostable reaction with $f(u)-f'(0)u=O(u^{1+\delta})$ (for some $\delta>0$). 
This also implies exponential propagation for general $1$-monostable reactions, albeit with the lower and upper exponential rates being $\frac{\gamma}{d+2s}$ and $\frac{\gamma'}{d+2s}$, respectively.   

An interesting feature of the  results in \cite{cabre2013influence,coulon2017exponential} is that, for Fisher-KPP reactions and $s<1$, the exponential propagation rates  for localized initial data differ from those for front-like data, and they also depend on the dimension.  These phenomena happen neither when $s=1$ (for any reaction), because the diffusion kernel is short range, nor for bistable reactions and $s\in(0,1)$, when ballistic propagation from localized data at the same speed as from front-like data follows from existence of traveling fronts.  

It is therefore not obvious which propagation rates one should expect for ignition and $\alpha$-monostable reactions with $\alpha>1$ when $s\in(0,1)$ and initial data are localized.  
One obviously has the ballistic lower bound \eqref{1.66}, and the same upper bounds as for front-like data (which follow immediately by comparison).  However, we are not aware of other relevant prior results  for \eqref{1.0} in this setting.  Our second main theorem therefore appears to provide the first non-trivial such result, and is again also qualitatively optimal.  It
shows that in all the cases from Theorem \ref{T.1.1}, propagation rates for localized data do coincide with those for front-like data.  In particular, unlike for 1-monostable reactions, they do not depend on the dimension.

\begin{theorem} \lb{T.1.3}
Let $0\le u\le 1$ be a solution to \eqref{1.0} such that \eqref{1.11a'} holds for some $\theta,R>0$ and  large enough $R'$  (depending on $f,s,\theta$), and let ${\underline x}_\lambda(t;u)$ and ${\overline x}_\lambda(t;u)$ be from \eqref{1.55a'} and \eqref{1.55b'}.  Then both parts of Theorem \ref{T.1.1} hold.
\end{theorem}

{\it Remark.}  We can obviously again replace $u(0,\cdot) \leq \chi_{B_R(0)}$ by $\limsup_{|x|\to\infty} |x|^{-2s}u(0,x)<\infty$ in (i), and by $\limsup_{|x|\to\infty} |x|^{-\frac{2s}\alpha}u(0,x)<\infty$ in (ii).  
\smallskip

The proofs of Theorems \ref{T.1.1} and \ref{T.1.3} rest on finding appropriate sub- and supersolutions $\Phi$ satisfying ${\underline x}_\lambda(t;\Phi)=O(t^\beta)={\overline x}_\lambda(t;\Phi)$, with $\beta=\frac 1{2s}$ for ignition reactions and $\beta=\frac{\alpha}{2s(\alpha-1)}$ for $\alpha$-monostable reactions.
Since these will accelerate in time, one cannot use the traveling front ansatz $\Phi(t,x)\sim \varphi(x-ct)$ in their construction, because then their transition regions (where they decrease from values close to $1$ to those close to $0$) would only travel with constant speed $c$.
One might instead hope to have $\Phi(t,x)\sim\varphi(x-ct^\beta)$,
which does travel with the right speed $O(t^{\beta-1})$.  However, it turns out that the acceleration of propagation also forces sub- and supersolutions to have  transition regions that stretch in time (see also \cite{garnier2017transition}).

We will therefore construct localized subsolutions of the form $\Phi(t,x)\sim \varphi(ct^{-\frac 1{2s}} x)$ in the proof of Theorems~\ref{T.1.1}(i) and \ref{T.1.3}(i).  Such functions propagate with speeds $O(t^{\frac{1}{2s}-1})$ but also ``flatten'' in space, having transition regions of widths $O(t^\frac 1{2s})$ (so the latter stretch with roughly the same speeds $O(t^{\frac{1}{2s}-1})$).  This will be sufficient for subsolutions, but we will have to employ a much more complicated construction for front-like supersolutions.  Their propagation speeds will again be $O(t^{\frac{1}{2s}-1})$, but we will need their
stretching speeds to also depend on the value of $\Phi$, and they will in fact grow from $O(t^{\frac{1}{2s}-2})$ where $\Phi\ge\theta_0$ to $O(t^{\frac{1}{2s}-1})$ where $\Phi\sim 0$.  Of course, this stretching will then accumulate over time to transition regions between different values of $\Phi$ having lengths from $O(t^{\frac{1}{2s}-1})$ to $O(t^{\frac{1}{2s}})$, meaning that these supersolutions will be flattened in a spatially non-uniform manner.
This approach, which seems to be necessary in the hunt for qualitatively optimal supersolutions in the ignition case (and hence  optimal upper bounds in Theorems \ref{T.1.1}(i) and \ref{T.1.3}(i)), makes this effort significantly more challenging and explains the complexity in our construction in Section \ref{S3} below.

This type of construction appears to be new, as all previous ones that we are aware of involve spatially uniform stretching rates.  In particular, the supersolutions for $\alpha$-monostable reactions  with $\alpha>1$ constructed in \cite{coville2020propagation}, which propagate with  optimal speeds $O(t^{\frac{\alpha}{2s(\alpha-1)}-1})$, have transition regions that stretch with speeds $O(t^{\frac{\alpha}{2s(\alpha-1)}-2})$.  The subsolutions we construct below in this case will have the same propagation and stretching speeds.  
However, unlike for ignition reactions, we are only able to achieve these optimal speeds with localized but not compactly supported  subsolutions.  This, and the fact that we need to find them in all dimensions $d\ge 1$, further complicate this part of our work.

Finally, we note that  smoothing  properties of the fractional parabolic dynamic of \eqref{1.0} mean that the sense in which our functions solve the PDE is not consequential here.  While we consider below mild solutions with uniformly continuous initial data, Theorem~\ref{T.2.1} shows that these immediately become classical.  This is also true for  bounded weak solutions, via an argument as in the proof of Theorem \ref{T.2.1} (based on the regularity results in \cite{fernandez2017regularity,KS14}), so  these three notions of solutions coincide here.  This means that our propagation results hold as well for not necessarily uniformly continuous initial data, due to the comparison principle.  Since we were not able to locate a suitable version of the latter in the literature, we prove it in Theorem \ref{T.2.2} below (which is hence of independent interest).  We in fact state it for distributional sub- and supersolutions (see Definition \ref{D.2.0'}) because the supersolutions we construct here will only be Lipschitz continuous.  

We also highlight here Lemma \ref{L.2.4} below, which constructs compactly supported stationary subsolutions to \eqref{1.0}.  These then provide initially compactly supported time-increasing solutions, which can be very convenient in the analysis of long-time dynamics of solutions (and specifically, construction of subsolutions in Section \ref{S4}).  We are not aware of such a result for \eqref{1.0} with $s\in(0,1)$ prior to our work, although its $s=1$ version is well known, and in \cite{BCHV} it was also obtained for diffusion operators with integrable kernels.

\smallskip
{\it Remark.}
Shortly before we finished  writing this paper in May 2021, we informed E.~ Bouin, J.~Coville, and G.~Legendre about it.   They posted the preprints \cite{BCL1,BCL2} on arXiv immediately afterwards, just days before we posted ours.  The main results claimed in \cite{BCL1,BCL2}  correspond to the first inequalities in  Theorem~\ref{T.1.1}(i,ii), respectively, which are our optimal lower bounds for front-like initial data
(both \cite{BCL1,BCL2} consider more general diffusion kernels in one dimension, with $x^{-1-2s}$ decay at $\pm\infty$). Unlike our constructions in Sections \ref{S4} and \ref{S5} below, the subsolution candidate functions presented in \cite{BCL1,BCL2} are front-like, so  they would not yield localized initial data results such as Theorem~\ref{T.1.3}   (even when $d=1$ because the diffusion kernels are long range for all $s\in(0,1)$).  
However, the 23-page May 2021 version of \cite{BCL2} is incomplete,
and was replaced in July 2022 by a 45-page version with a much longer proof containing many changes and additions.  Moreover, the May 2021 version of \cite{BCL1} is clearly very preliminary and no other version seemed to be available at the time the present paper went into press in July 2023.  Neither preprint appears to have been peer reviewed by that time either.

\medskip
{\bf Organization of the Paper and Acknowledgements.}
In Section \ref{S2} we collect various preliminary results, including a comparison principle for  \eqref{1.0}.  We then prove parts (i) of Theorems \ref{T.1.1} and \ref{T.1.3} in Sections \ref{S3} and \ref{S4}, and parts (ii) in Section \ref{S5} (these three sections are completely independent and can be read in any order). 

AZ acknowledges partial support by  NSF grant DMS-1900943 and by a Simons Fellowship.

\section{Well-posedness and a Comparison Principle}\lb{S2}


In this section we collect some basic well-posedness and regularity results for \eqref{1.0}.
We also prove two important (and to the best of our knowledge new) results here.
The first is a comparison principle,  Theorem~\ref{T.2.2}, which  removes certain restrictive hypotheses from previous results (see the paragraph before Definition \ref{D.2.0'}).  The second is Lemma \ref{L.2.4}, which constructs initially compactly supported time-increasing solutions to \eqref{1.0} for ignition (and therefore also for $\alpha$-monostable) reactions and all $s\in(0,1)$.

We start with the notion of mild solutions, defined via Duhamel's formula (see \cite{cabre2013influence,pazy}). We use $C_{b,u}(X)$ to denote the space 
of bounded uniformly continuous functions on $X$ (with the supremum norm), and $S_t$ to denote the semigroup generated by $(-\Delta)^s$ on $\bbR^d$.

\begin{definition} \lb{D.2.0}
We say that $u\in C([0,T);C_{b,u}(\bbR^d))$ is a  {\it mild solution} to \eqref{1.0} (and that it is {\it global} if $T=\infty$) if for each $t\in [0,T)$ we have
\[
u(t,\cdot)=S_t [u(0,\cdot)]+\int_0^t S_{t-\tau} [f(u(\tau,\cdot))] d\tau.
\]
\end{definition}

{\it Remark.}
Notice that if $u\in C([T_0,T_1];C_{b,u}(\bbR^d))$ for some $T_0<T_1$, then it is also uniformly continuous in time on $[T_0,T_1]$ and it follows that in fact $u\in C_{b,u}([T_0,T_1] \times \bbR^d)$.
\smallskip

We now have the following global well-posedness result.

\begin{theorem} \lb{T.2.0}
If $s\in(0,1)$ and $f$ is Lipschitz continuous, then
for any $u_0\in C_{b,u}(\bbR^d)$, there is a unique global mild solution $u$
to \eqref{1.0} with $u(0,\cdot)=u_0$.
\end{theorem}


{\it Remark.} 
Theorem \ref{T.2.0} can be proved via a standard fixed point argument using that 
\[
N[u] (t,\cdot):=S_t [u(0,\cdot)]+\int_0^t S_{t-\tau} [f(u(\tau,\cdot))] d\tau
\]
defines a contraction mapping $N$ on the subspace of $u\in C([0,T];C_{b,u}(\bbR^d))$ with $u(0,\cdot)=u_0$ when $T$ is sufficiently small, see for instance \cite{cabre2013influence,lim17} (both these papers concern more general diffusion operators than $(-\Delta)^s$).  We note that although $f'\in C_{b,u}(\bbR)$ is assumed in \cite[Sections 2.3 and 2.4]{cabre2013influence}, this can easily be relaxed to $f$ being Lipschitz (we prefer to consider here this case instead of $f\in C^1([0,1])$). 
We also mention that a viscosity-solutions-based approach to well-posedness, via maximum principles for general non-local nonlinear PDEs and Perron's method, was used in \cite{caffarelli2011regularity,cardaliaguet2011holder,jakobsen2005continuous}.  
\smallskip

We next turn to a comparison principle for \eqref{1.0}, and the related definition of sub- and supersolutions.  
 We were not able to use in this work comparison principles that we found in the literature, as these do not quite apply to the Lipschitz continuous sub- and supersolutions we construct below  (mild solutions have better regularity, see Theorem \ref{T.2.1}).  For instance 
 \cite[Proposition 2.8]{cabre2013influence} only applies to classical sub- and supersolutions that satisfy an extra hypothesis on their order as $|x|\to \infty$ at all times $t\ge 0$,  while \cite[Proposition 2.11]{cabre2013influence} only applies to mild solutions to $u_t+(-\Delta)^su=h(t,x)$ with uniformly continuous $h$. 
We therefore prove here a comparison principle without extra hypotheses and in the more general distributional sense, which then also applies to mild solutions due to Remark 2  below.



\begin{definition} \lb{D.2.0'}
We say that  $u\in C((T_0,T_1); C_{b,u}(\mathbb{R}^d))$ is a {\it subsolution (supersolution)} to \eqref{1.0} if for each $0\le \varphi\in C_c^\infty((T_0,T_1)\times\bbR^d)$ we have
\[
\int_{T_0}^{T_1} \int_{\bbR^d} \left[ - u(t,x)\varphi_t(t,x)+u(t,x)(-\Delta)^s\varphi(t,x)-f(u(t,x))\varphi(t,x) \right]\, dxdt\leq 0 \quad (\geq  0).
\]
\end{definition}

{\it Remarks.}
1.  Note that when $s\in(0,\frac 12)$, any bounded Lipschitz continuous function $u$ has bounded $(-\Delta)^{s} u$; and if it satisfies 
$ u_t+(-\Delta)^{s} u-f(u) \leq 0$ ($\geq  0$) for a.e.~$(t,x)$, then it is clearly a subsolution (supersolution) to \eqref{1.0}.
\smallskip

2.  It is easy to show that a mild solution $u$ to \eqref{1.0} on time interval $[0,T)$ is both a sub- and a supersolution on time interval $(0,T)$. Indeed, let $u_\eps:=\phi_\eps*u$, with $\phi_\eps$ a smooth space-time mollifier as in the following proof.  Then $u_\eps$ is a classical solution to $u_t+(-\Delta)^su=f_\eps(t,x)$ on time interval $(T_0+\eps,T_1-\eps)$, where $f_\eps:=\phi_\eps*(f\circ u)$. Since $u_\eps\to u$, and $f_\eps\to f\circ u$ uniformly on $[t_0,t_1]\times\bbR^d$ for any  $[t_0,t_1]\subseteq (T_0,T_1)$ (see the remark after Definition \ref{D.2.0}), this
yields the claim.
\smallskip




\begin{theorem}\lb{T.2.2}
Let $s\in (0,1)$ and $f$ be Lipschitz continuous.  If $u,v\in C([0,T); C_{b,u}(\mathbb{R}^d))$ are, respectively, a subsolution and a supersolution to \eqref{1.0} on time interval $(0,T)$ and satisfy $u(0,\cdot)\leq v(0,\cdot)$ on $\bbR^d$, then $u\leq v$ on $[0,T)\times\bbR^d$.
\end{theorem}

\begin{proof}
It {suffices} to prove that $u\le v$ on $[0,T']\times\bbR^d$ for any $T'<T$, so we can assume that $T<\infty$ and  $u,v\in C([0,T]; C_{b,u}(\mathbb{R}^d))$.
Let $K:=\max\{\|f'\|_\infty,\|u\|_\infty,\|v\|_\infty\}$, and then let
\[
\tilde{u}(t,x):=e^{Kt}u(t,x),\qquad \tilde{v}(t,x):=e^{Kt}v(t,x),\qquad g(t,u):=K{u}+e^{Kt}f(e^{-Kt}{u}).
\]
Then  $\tilde{u}$ and $\tilde{v}$ are, respectively, a subsolution and a supersolution to \eqref{1.0} with $f(u)$ replaced by $g(t,u)$,  on time interval $(0,T)$. 
Moreover, $\max\{\|\tilde u\|_\infty,\|\tilde v\|_\infty\}\le Ke^{KT}$, and 
if we let $S:=[0,T]\times [-Ke^{KT},Ke^{KT}]$, then 
for all $(t,u)\in S$ we have
\beq\lb{3330}
0\le g_u(t,u)\le 2K \qquad\text{ and }\qquad |g_t(t,u)|\leq Ke^{KT} \|f|_{[-K,K]}\|_{{\infty}}+K^3=:K'.
\eeq
Finally, we have  $\tilde{u}(0,\cdot)\leq \tilde{v}(0,\cdot)$, and proving $u\leq v$ is equivalent to proving $\tilde{u}\leq \tilde{v}$. 
We will therefore slightly abuse notation, and write  below $u,v$ instead of $\tilde{u}, \tilde{v}$.

For any small $\eps\in (0,1)$, fix a smooth mollifier $\phi_\eps\ge 0$ with $\supp\, \phi_\eps \subseteq B_\eps(0) \subseteq \bbR^{d+1}$ and $\int_{\bbR^{d+1}}\phi_\eps(t,x)d(t,x)=1$.
Then for $(t,x)\in[0,T-2\eps]\times\bbR^d$ let 
\[
u_\eps(t,x):=(\phi_\eps*u)(t+\eps,x)\qquad\text{and}\qquad v_\eps(t,x):=(\phi_\eps*v)(t+\eps,x)
\]
(the $\eps$-shifts in time allow us to define $u_\eps,v_\eps$ at $t=0$).
The remark after Definition \ref{D.2.0} shows that there is $\omega_\eps\geq 0$ such that $\lim_{\eps\to0^+}\omega_\eps=0$ and
\beq\lb{3332}
\sup_{\substack{(t,x),(\tau,y)\in [0,T]\times\bbR^d,\\
\max\{ |t-\tau|, |x-y|\}\leq 2\eps}} \max\{ |u(t,x)-u(\tau,y)|, |v(t,x)-v(\tau,y)|\} \leq \omega_\eps.
\eeq
Then on $[0,T-2\eps]\times\bbR^d$ we have
\beq\lb{3333}
\max\{ |u_\eps-u|,|v_\eps-v| \} \leq \omega_\eps.
\eeq

Now consider any $(t,x)\in [0,T-2\eps]\times\bbR^d$.  If there is $(t',x')\in B_\eps(t+\eps,x)$ such that $u(t',x')\leq v(t',x')$, then it follows from \eqref{3330} and \eqref{3332} that
\begin{align*}
(\phi_\eps* [g(\cdot, u(\cdot,\cdot))  - g(\cdot, v(\cdot,\cdot)) ] ) & \, (t+\eps,x) \\
&\leq g(t',u(t',x'))-g(t',v(t',x')) 
+ 4\eps \|g_t |_{S} \|_{\infty} + 2\omega_{\eps} \|g_u |_{S}\|_{\infty}  \\
&\leq 4\eps K' + 4K\omega_{\eps}.
\end{align*}
If instead $u\geq v$ on $B_\eps(t+\eps,x)$, then $g_u\le 2K$ yields
\[
(\phi_\eps* [g(\cdot, u(\cdot,\cdot))  - g(\cdot, v(\cdot,\cdot)) ] ) \, (t+\eps,x)\leq 2K(\phi_\eps * (u-v))(t+\eps,x)\leq 2K(u_\eps(t,x)-v_\eps(t,x)).
\]
From these estimates, and from $u$ and $v$ being, respectively, a sub- and a supersolution, we get for $w_\eps:=u_\eps-v_\eps$ and $\omega'_\eps:=4\eps K' + 4K\omega_{\eps}$ ($\to 0$ as $\eps\to 0$) that
\[
h_\eps:=(w_\eps)_t+(-\Delta)^s w_\eps\leq (\phi_\eps*[g(\cdot, u(\cdot,\cdot))  - g(\cdot, v(\cdot,\cdot)) ]) \,(\cdot+\eps,\cdot)\leq  \max\{2K w_\eps,\omega'_{\eps}\}.
\]
Duhamel's principle for smooth solutions to the linear PDE 
\[
w_t+(-\Delta)^s w=h(t,x)
\]
 (see \cite{cabre2013influence,pazy,lim17}) now yields
\[
w_\eps(t,x)\leq S_t [w_\eps(0,\cdot)](x)+\int_0^t S_{t-\tau} [\max\{2K w_\eps(\tau,\cdot),\omega'_{\eps} \} ](x)\,d\tau
\]
for all $(t,x)\in [0,T-2\eps]\times\bbR^d$.
Since $S_t$ preserves order and $S_t[1]\equiv 1$, we have $S_t[w_\eps(0,\cdot)]\leq 2\omega_\eps$ on $\bbR^d$ (by $u(0,\cdot)\leq v(0,\cdot)$ and \eqref{3333}), and if we let $\xi(t):=\max\{\sup_{x\in\bbR^d}w_\eps(t,x),0\}$, then
\[
\xi(t)\leq 2\omega_\eps+T\omega'_{\eps}+ \int_0^t 2K \xi(\tau)\, d\tau
\]
for each $t\in [0,T-2\eps]$. Gr\" onwall's inequality now yields
\[
\xi(t) \leq (2\omega_\eps+T\omega'_{\eps})\,e^{2KT}
\]
for $t\in [0,T-2\eps]$, hence
\[
\limsup_{\eps\to 0} \sup_{(t,x)\in  [0,T-2\eps]\times\bbR^d} w_\eps(t,x) = 0.
\]
This shows that $u\leq v$ on $[0,T]\times\bbR^d$, finishing the proof. 
\end{proof}

Similarly to parabolic PDE with classical diffusion, the dynamics of \eqref{1.0} provides certain smoothing, which is the basis of relevant regularity results.  The following theorem, in which we suppress dependence of all constants on $d$ in the notation, is consequence of results from \cite{KS14,fernandez2017regularity} (we also refer the reader to \cite{cardaliaguet2011holder,caffarelli2009regularity,jin2015schauder} for the viscosity solutions setting). 

\begin{theorem}\lb{T.2.1}
Let $s\in (0,1)$ and $f$ be Lipschitz continuous, and let $u\in C([0,T);C_{b,u}(\bbR^d))$ be a bounded mild solution to \eqref{1.0}.  There is $\sigma=\sigma(s)>0$ such that for any $\tau\in (0,T)$ there is $C=C(s,f,\|u\|_\infty,\tau)>0$ such that
\[
\|u\|_{ C^{1+\sigma/2s,2s+\sigma}([\tau,T)\times\bbR^d)}\leq C.
\]
In particular, $u$ is a classical solution to \eqref{1.0}.
\end{theorem}

\begin{proof}
For any $\eps\in(0, \frac{\tau}{4})$, let $\phi_\eps$ be the space-time mollifier from the proof of Theorem \ref{T.2.2}. If $f_\eps:=\phi_\eps*(f\circ u)$, then $u_\eps:=\phi_\eps*u$ satisfies  in the classical sense the linear PDE
\beq\lb{l251}
(u_\eps)_t+(-\Delta)^su_\eps=f_\eps(t,x)
\eeq
on  $(\frac{\tau}{4},T-\eps)\times\bbR^d$.
Since $\|u\|_\infty<\infty$, we have
\[
\max\{\|u_\eps\|_\infty, \|f_\eps\|_\infty \} \leq C_1:= |f(0)|+(1+\|f'\|_\infty)(1+\|u\|_\infty).
\]
The interior H\"{o}lder estimate \cite[Theorem 1.1]{KS14} now yields  $\sigma=\sigma(s)\in(0,\min\{1,2s\})$ and $C_2=C_2(s,C_1,\tau)>0$ such that
\[
\|u_\eps\|_{C^{{\sigma/2s},\sigma} ([{\tau}/{2},T-\eps)\times\bbR^d )}\leq C_2.
\]
Since $u_\eps\to u$ uniformly on $[\frac{\tau}{2},T-\delta]\times\bbR^d$ for any $\delta>0$ by the remark after Definition \ref{D.2.0}, taking $\eps\to 0$ shows that
\[
\|u\|_{C^{{\sigma/2s},\sigma}([{\tau}/{2},T)\times\bbR^d)} \le C_2.
\] 
This and Lipschitz continuity of  $f$ yield $C_3=C_3(C_1,C_2)>0$ such that
\[
\|f_\eps\|_{C^{{\sigma}/{2s},\sigma}([{\tau}/{2},T-\eps)\times\bbR^d) } \le C_3
\]
for all $\eps\in(0,\frac{\tau}{4})$. It now follows from \eqref{l251}  and  \cite[Theorem 1.1]{fernandez2017regularity} that 
\[
\|u_\eps\|_{C^{1+{\sigma}/{2s},2s+\sigma}([\tau,T-\eps)\times\bbR^d)}\leq C_4
\]
for these $\eps$, with $C_4=C_4 (s,C_1,C_2,\tau)>0$. Taking $\eps\to 0$ finishes the proof.
\end{proof}

Finally, to obtain the lower bounds in Theorems \ref{T.1.1} and \ref{T.1.3}, we will need to use certain non-trivial initial data $ u_0:\bbR^d\to[0,1]$ satisfying
\beq\lb{2.18}
-(-\Delta)^s u_0 +f(u_0)\geq 0
\eeq
(note that the comparison principle then shows that  solutions to \eqref{1.0} with such initial data are time-increasing). The following lemma, whose proof we postpone  to the appendix, provides such functions for ignition reactions (see the remark below for monostable reactions).  

\begin{lemma}\lb{L.2.4}
Let $f$ be an ignition reaction with ignition temperature $\theta_0\in (0,1)$, and let $s\in(0,1)$. For  $\theta\in (\theta_0,1)$, there are $R_\theta\geq 1$
and a non-increasing  
smooth function $u_\theta$ on $\bbR$ (both depending also on $s,f,d$) such that 
\begin{gather}
\theta \chi_{(-\infty,0]}\leq  u_\theta\leq \theta \chi_{(-\infty,R_\theta]},  \lb{2.20} \\
\inf_{x\in\bbR} \left[ -(-\partial_{xx})^su_\theta(x)+f(u_\theta(x)) \right] \ge 0,  \lb{2.18''} \\
\inf_{x\le R_\theta} \left[ -(-\partial_{xx})^su_\theta(x)+f(u_\theta(x)) \right] > 0,  \lb{2.18'}
\end{gather}
and the function $\bar u_\theta(x):=u_\theta(|x|)$ on $\bbR^d$ satisfies
\begin{gather}
\inf_{x\in\bbR^d} \left[ -(-\Delta)^s \bar u_\theta(x)+f(\bar u_\theta(x)) \right] \ge 0,  \lb{2.18''x} \\
\inf_{|x|\le R_\theta} \left[ -(-\Delta)^s \bar u_\theta(x)+f(\bar u_\theta(x)) \right] > 0.  \lb{2.18'x}
\end{gather}
Moreover, if $0\le u\le 1$ is a global mild solution to \eqref{1.0} and $u(0,\cdot)\ge  \bar u_\theta(\cdot-x_0)$ for some $x_0\in\bbR^d$, then $u(t,\cdot)\to 1$ locally uniformly on $\bbR^d$ as $t\to\infty$.
%
%
\end{lemma}

{\it Remark.}  
Since for each monostable reaction $f$ and each $\theta\in(0,1)$, there is an ignition reaction $g\le f$ with ignition temperature smaller than $\theta$, the lemma extends to such $f,\theta$.


\section{Supersolutions for Ignition Reactions}\lb{S3}



In this section we prove the upper bound in Theorem \ref{T.1.1}(i), which then automatically  provides the same bound in Theorem \ref{T.1.3}(i) via the comparison principle. The main step is construction of a family of supersolutions $\Phi$ to \eqref{1.1} that satisfy ${\overline x}_\lambda(t;\Phi)=O(t^{\frac{1}{2s}})$.

We start with a simple fractional Laplacian estimate.   For $A_1<A_2$ and $\theta\in (0,1)$, let
\[
\psi_{A_1,A_2,\theta}(x):=\left\{
\begin{aligned}
    &1 &&\text{ if }x\leq A_1,\\
    &1- (1-\theta) \frac{x-A_1}{A_2-A_1} &&\text{ if }x\in ( A_1,A_2],\\
    &\theta &&\text{ if }x> A_2.
\end{aligned}\right.
\]
Note that $s<\frac{1}{2}$ and Lipschitz continuity of $\psi_{A_1,A_2,\theta}$ show that $(-\partial_{xx})^{s}\psi_{A_1,A_2,\theta}$ is bounded. 

\begin{lemma}\lb{L.6.1}
Let $s\in(0,\frac 12)$, $C_{\alpha}:= \max\{\frac {c_s}{2s(1-2s)},1\}$, $A_1<A_2$, and $\theta\in (0,1)$.  If $\varphi\leq\psi_{A_1,A_2,\theta}$ and  $\varphi(x)=\psi_{A_1,A_2,\theta}(x)$ for some $x\in \bbR$, then
\[
(-\partial_{xx})^s\varphi(x)\geq-C_s(1-\theta) (A_2-A_1)^{-2s}.
\]
\end{lemma}

\begin{proof}
By \eqref{1.2x}, it suffices to assume that $\varphi\equiv \psi_{A_1,A_2,\theta}$. Then
\[
\varphi(x)-\varphi(x+h)\geq \varphi(A_2)-\varphi(A_2+h)
\]
for each $h\in\bbR$, so 
it suffices to assume that $x=A_2$.  Now a direct computation yields
\begin{align*}
    (-\partial_{xx})^s\varphi(A_2)
    &=-c_s(1-\theta)\int_{-\infty}^{A_1-A_2} |h|^{-1-2s} dh- \frac{c_s(1-\theta)}{A_2-A_1}\int_{A_1-A_2}^{0} |h|^{-2s}dh\\
    &= -c_s(1-\theta) \left((2s)^{-1}+(1-2s)^{-1}\right)(A_2-A_1)^{-2s},
\end{align*}
finishing the proof.
\end{proof}

We will now construct an infinite family of supersolutions to \eqref{1.1} indexed by $k\in\bbN$ (see \eqref{Phi} below), each defined on a finite time interval of length $e^{e^{O(k)}}$ and obtained by gluing together $k+3$ separate pieces (all but one of them linear in space).  See the introduction for a discussion of the reasons for such a complicated construction.

Let us take any  $k\in\bbN$,  and for any $n\in\bbN_k:=\{0,1,...,k\}$ let
\[
\alpha^k_n:=\Sigma_{j=1}^{n} (k-n+j){(2s)^{j-1}}\qquad \text{ and }\qquad \beta^k_n:=2^{-\alpha_n^k},
\]
where $\sum_{j=1}^0 a_j:=0$ (so $\alpha^k_0=0$ and $\beta^k_0=1$).
We then have $\alpha^k_k\leq \sum_{j=1}^\infty j(2s)^{j-1} = \frac 1{(1-2s)^{2}}$, so
\beq\lb{6.4}
\beta^k_k\geq 2^{-\frac{1}{(1-2s)^2}}.
\eeq
(In fact, $\alpha^k_{k-n}$ is increasing in $k$ for each fixed $n\ge 0$ and converges to $\frac 1{(1-2s)^{2}} + \frac n{1-2s}$.)
Also, from $\alpha_n^k\geq k-n+1$ we have
\beq\lb{6.5}
\Sigma_{n=1}^k\beta_n^k\leq 1,
\eeq
and a simple computation yields
\beq\lb{6.5a}
\beta_n^k = 2^{-k+n-1} (\beta_{n-1}^k)^{2s}.
\eeq

We also have the following simple lemma.

\begin{lemma}\lb{L.6.2}
If  $t\geq 2^{\frac{k}{1-2s}}$, then for any $n\in\bbN_{k}\setminus\{0\}$,
\beq\lb{6.1}
\beta^k_{n} t^{\frac{1}{2s}-(2s)^{n}} \ge 2\beta^k_{n-1} t^{\frac{1}{2s}-(2s)^{n-1}}.
\eeq
\end{lemma}

\begin{proof}
A direct computation shows that \eqref{6.1} is equivalent to
\[
t^{(2s)^{n-1}-(2s)^{n}}\geq 2^{k(2s)^{n-1}-\Sigma_{j=1}^{n-1}(2s)^{j-1}}.
\]
This is  equivalent to
\[
t^{1-2s}\geq 2^{k-((2s)^{1-n}-1)/(1-2s)},
\]
which clearly holds by the hypothesis.
\end{proof}


In the rest of this section, and in the next section, we will assume the following.

\smallskip
\begin{itemize}
    \item[ (I)] Let $f$ be an ignition reaction with ignition temperature $\theta_0\in (0,1)$, and let $s\in(0,\frac{1}{2})$.
\end{itemize}
\smallskip

Let us now define $\theta_*:=\frac{\theta_0}{2}$.
For $n\in\bbN_k$  we let
\[
\theta_n^k:=(1-2^{-k+n-1})\theta_* \qquad\text{and}\qquad \theta_{-1}^k:=1,
\]
and then for $t\ge 1$
\[
 l^k_{n}(t):=\Sigma_{j=0}^{n}\beta_j^k\,t^{\frac{1}{2s}-(2s)^j}  \qquad\text{and}\qquad l_{-1}^k(t):=0
\]
as well as
\[
L^k_n(t,x):=\theta_{n-1}^k-(\theta_{n-1}^k-\theta_n^k)\frac{x-l^k_{n-1}(t)}{l^k_{n}(t)-l^k_{n-1}(t)}.
\]
Therefore, the function
\[
\hat{L}^k(t,x):=
    L^k_n(t,x) \qquad \text{ when }x\in ( l^k_{n-1}(t),l^k_{n}(t)]\text{ for some }n\in\bbN_k
\]
is continuous on its domain  $\left\{(t,x)\in \bbR^2 \,| \, t\geq 1 \text{ and } x\in  (0,l^k_{k}(t)]\right\}$, and piecewise linear in $x$ for each fixed $t\ge 1$.

Finally, with $C_s$  from Lemma \ref{L.6.1}, let 
\beq\lb{csharp}
c_* :=\max\left\{ C_s+ 2 {\|f\|_\infty}, 
(C_s\gamma_{0})^{\frac{1}{s}},  (C_s \gamma_{0} \gamma_{1}^{2s} )^2  \right\},
\eeq
where 
\[
\gamma_{0}:=(4\theta_*^{-1})^{\frac 1{s}} \qquad \text{and} \qquad \gamma_{1}:= 2^{1+\frac{1}{(1-2s)^2}},
\]
and define the Lipschitz continuous function
\beq\lb{Phi}
\Phi^k(t,x):=\left\{
\begin{aligned}
&1 && x\leq c_*  t^\frac{1}{2s},\\
&\hat{L}^k(t,x-c_*  t^\frac{1}{2s}) && x\in ( c_*  t^\frac{1}{2s},c_*  t^\frac{1}{2s}+l^k_{k}(t)],\\
& \left[ (\theta_k^k)^{-\frac{1}{2s}}+ c_* ^{-\frac{1}{2}}t^{-\frac{1}{2s}} \left( x-c_*  t^\frac{1}{2s}-l^k_{k}(t) \right)  \right]^{-2s} && x> c_*  t^\frac{1}{2s}+l^k_{k}(t).
\end{aligned}\right.
\eeq

We now show that $\Phi^k$ is a supersolution to \eqref{1.1} on some (long for large $k$) time interval.

\begin{theorem}\lb{T.3.3'}
Let $f$ and $s$ satisfy (I), and let $\Phi^k$ be from \eqref{Phi} for each $k\in\bbN$. Then $\Phi^k$ is a supersolution to \eqref{1.1} on the time interval $(2^\frac{k}{1-2s},2^{(2s)^{-k}})$.
\end{theorem}

\begin{proof}
We write $\Phi=\Phi^k$ for simplicity. Let $T_k:=2^{(2s)^{-k}}$ and fix any $t\in [2^\frac{k}{1-2s},T_k]$.
At any $x< c_*  t^\frac{1}{2s}$ we clearly have $\Phi_t=0\le (-\partial_{xx})^s\Phi$, and so $\Phi_t+(-\partial_{xx})^s\Phi-f(\Phi)\geq 0$ because  $f(1)=0$.  By Remark 1 after Definition \ref{D.2.0'}, it suffices to extend this claim to a.e.~$x> c_*  t^\frac{1}{2s}$.

Next we claim that $\Phi(t,\cdot)$ is convex on $[c_*  t^\frac{1}{2s},\infty)$.
The slope of $L_n^k(t,\cdot)$ for $n\in \bbN_{k}\setminus\{0\}$ is
\[
-\frac{\theta_{n-1}^k-\theta_n^k}{l^k_{n}(t)-l^k_{n-1}(t)}=-\frac{2^{-k+n-2}\theta_*}{\beta_n^k\,t^{\frac{1}{2s}-(2s)^n}},
\]
while for $n=0$ it is $-(1-\theta_0^k)t^{1-\frac 1{2s}}$.
Since $t\geq 2^\frac{k}{1-2s}$ and $t^{1-2s}\ge 1\ge \frac{\theta_*}{2(1-\theta_*)}$,  Lemma \ref{L.6.2} shows that $\Phi(t,\cdot)$ is convex on $[c_*  t^\frac{1}{2s},c_*  t^\frac{1}{2s}+l^k_{k}(t)]$.  
It is clearly also convex on $[c_*  t^\frac{1}{2s}+l^k_{k}(t),\infty)$, so we only need to check one-sided derivatives at  $y:=c_*  t^\frac{1}{2s}+l^k_{k}(t)$. 
We have
\beq\lb{6.2}
\lim_{x\to y+}\Phi_x(t,x)=-2s(\theta^k_k)^{1+\frac{1}{2s}}c_* ^{-\frac{1}{2}}t^{-\frac{1}{2s}}=-2s(2^{-1}\theta_*)^{1+\frac{1}{2s}}c_* ^{-\frac{1}{2}}t^{-\frac{1}{2s}},
\eeq
as well as (using also \eqref{6.5})
\[
\lim_{x\to y-}\Phi_x(t,x)=-\frac{\theta_{k-1}^k-\theta_k^k}{l^k_{k}(t)-l^k_{k-1}(t)}=-\frac{\theta_*}4 (\beta_k^k)^{-1}\,t^{-\frac{1}{2s}+(2s)^k} \leq -\frac{\theta_*}4 t^{-\frac{1}{2s}},
\]
which is no more than \eqref{6.2} because $s<\frac 12$ and so $c_* \geq 1\geq 16s^2(2^{-1}\theta_*)^\frac{1}{s}$. 
The claim follows.

Next consider any $x\in ( c_*  t^\frac{1}{2s},c_*  t^\frac{1}{2s}+l^k_0(t))$.
It follows from the definition of $\Phi$ that
\[
\Phi_t(t,x)=(1-\theta_0^k)(c_* + (1-2s)(2s)^{-1} xt^{-\frac{1}{2s}})\geq c_* (1-\theta_0^k).
\]
Lemma \ref{L.6.1} with $A_1:=c_*  t^\frac{1}{2s}$, $A_2:=c_*  t^\frac{1}{2s}+l^k_0(t)$, $\theta:=\theta^k_0$, 
together with $c_* \geq C_s+2{\|f\|_\infty}$, show that at such $x$ we have
\begin{align*}
\Phi_t+(-\partial_{xx})^s\Phi-f(\Phi)\geq c_* (1-\theta_0^k)-C_s(1-\theta_0^k) t^{-1+2s}-{\|f\|_\infty}\geq0 .
\end{align*}

Convexity of $\Phi(t,\cdot)$ on $[c_* t^{\frac12},\infty)$ shows that for any $n\in \bbN_k\setminus\{0\}$ and any $x\in ( c_*  t^\frac{1}{2s}+l^k_{n-1}(t),c_*  t^\frac{1}{2s}+l^k_{n}(t))$ we have
\[
\Phi(t,\cdot)\leq \psi_{c_*  t^{1/2s},x,\Phi(t,x)}
\]
on $\bbR$, and thus Lemma \ref{L.6.1} and the definition of $l^k_{n-1}(t)$  yield
\beq\lb{6.8}
(-\partial_{xx})^s \Phi(t,x)\geq -C_s(1-\theta_{n}^k)(x-c_*  t^\frac{1}{2s})^{-2s} 
\geq -C_s(\beta_{n-1}^k)^{-2s}\,t^{-1+(2s)^{n}}.
\eeq
Since $l^k_{n-1}(t)$ is increasing in $t\ge 1$ and $\beta_n^k\le 1$ by \eqref{6.5}, 
we obtain
\begin{align*}
\Phi_t(t,x)&=\frac{d}{dt}  L_n^k(t,x-c_* t^\frac{1}{2s})\\
&\geq 2^{-k+n-2}\theta_* (\beta_n^k)^{-1} \left[- t^{-\frac{1}{2s}+(2s)^n}\partial_t(-c_*  t^{\frac{1}{2s}}) - (x-c_*  t^\frac{1}{2s}-l^k_{n-1}(t)) \partial_t(t^{-\frac{1}{2s}+(2s)^n}) \right] \\
&\geq  2^{-k+n-2}\theta_*(\beta_n^k)^{-1} c_*  ({2s})^{-1} t^{-1+(2s)^n}.
\end{align*}
From this, \eqref{6.8}, \eqref{6.5a}, and $c_* \geq \frac{4sC_s}{\theta_*}$ (due to $c_* \ge (C_s\gamma_{0})^{\frac{1}{s}}$ and $C_s\ge 1$), we obtain $\Phi_t+(-\partial_{xx})^s\Phi \geq 0$  at the $x$ in question (notice that $f(\Phi(t,x))=0$ for all $x\ge c_*  t^\frac{1}{2s}+l^k_0(t)$).

Finally we need to consider any $x> c_*  t^\frac{1}{2s}+l^k_{k}(t)$, and this is the region where we will use that $t\leq T_k$.
Since $l^k_{k}(t)$ is increasing in $t\ge 1$, with $y_{t,x}:=x-c_*  t^\frac{1}{2s}-l^k_{k}(t)$ we have
\begin{align*}
    \Phi_t (t,x) &\geq \left[ (\theta_k^k)^{-\frac{1}{2s}}+c_* ^{-\frac{1}{2}}t^{-\frac{1}{2s}} y_{t,x}\right]^{-1-2s} \left(c_* ^{-\frac{1}{2}}t^{-\frac{1}{2s}-1}y_{t,x}+c_* ^{\frac{1}{2}}t^{-1} \right).
\end{align*}
Thus if $c_* ^{-\frac{1}{2}}t^{-\frac{1}{2s}} y_{t,x} \leq 1$, then $(\theta_k^k)^{-\frac{1}{2s}}+1\le 2 (\frac{\theta_*}2)^{-\frac{1}{2s}}$ and so by $s<\frac 12$ we have
\beq\lb{6.6}
\Phi_t (t,x) \geq  2^{-\frac{(1+2s)^2}{2s}}\theta_*^{\frac{1+2s}{2s}} c_* ^{\frac{1}{2}} t^{-1} \ge \gamma_0^{-1} c_* ^{\frac{1}{2}} t^{-1}.
\eeq
And if $c_* ^{-\frac{1}{2}}t^{-\frac{1}{2s}} y_{t,x}\geq 1$, we obtain
\beq\lb{6.7}
\Phi_t (t,x) \geq \left[ (\theta_k^k)^{-\frac{1}{2s}}+1\right]^{-1-2s} \left(c_* ^{-\frac{1}{2}}t^{-\frac{1}{2s}}y_{t,x} \right)^{-1-2s}c_* ^{-\frac{1}{2}} t^{-\frac{1}{2s}-1}y_{t,x}
\geq  \gamma_0^{-1} c_* ^{s} y_{t,x}^{-2s}.
\eeq

Since  $\Phi(t,\cdot)$ is convex on $[c_*  t^\frac{1}{2s},\infty)$, Lemma \ref{L.6.1} shows that
\beq\lb{6.10}
(-\partial_{xx})^s\Phi(t,x)\geq -C_s(x-c_* t^\frac{1}{2s})^{-2s}.
\eeq
Therefore, at all $x> c_*  t^\frac{1}{2s}+l^k_{k}(t)$ such that $c_* ^{-\frac{1}{2}}t^{-\frac{1}{2s}} y_{t,x} \ge 1$ we have
\[
\Phi_t+(-\partial_{xx})^s\Phi \ge (\gamma_0^{-1} c_* ^{s}-C_s) y_{t,x}^{-2s} \geq 0
\]
by \eqref{6.7} and \eqref{csharp}.  If instead $c_* ^{-\frac{1}{2}}t^{-\frac{1}{2s}} y_{t,x} \leq 1$, we note that $t\leq T_k$ implies $t^{(2s)^k}\leq 2$, so
\[
l^k_{k}(t)\geq \beta_k^k\,t^{\frac{1}{2s}-(2s)^k}\geq 2^{-\frac{1}{(1-2s)^2}-1}t^\frac{1}{2s}=\gamma_1^{-1}t^\frac{1}{2s}
\]
by  \eqref{6.4}. Hence \eqref{6.10} yields
\[
(-\partial_{xx})^s\Phi(t,x)\geq -C_s l^k_{k}(t)^{-2s}\geq -C_s \gamma_1^{2s}t^{-1},
\]
and then \eqref{6.6} and \eqref{csharp} again show that $\Phi_t+(-\partial_{xx})^s\Phi\geq 0$ at such $x$.
Therefore, $\Phi$ is indeed a supersolution to \eqref{1.1} on the time interval $(2^\frac{k}{1-2s},T_k)$. 
\end{proof}

We can now use the supersolutions from Theorem \ref{T.3.3'} to obtain an upper bound for general solutions to \eqref{1.1}.

\begin{theorem}\lb{T.3.3}
Let $f$ and $s$ satisfy (I), and let $0\le u\le 1$ solve \eqref{1.1}. If 
\[
u(0,x)\leq Ax^{-2s}
\]
 for some $A\geq 1$ and all $x>0$,
then for each $\lambda\in (0,1)$ there is $C_{\lambda,A}>0$ (depending also on $s,f$) such that 
for all $t\ge 0$ we have
\[
{\overline x}_\lambda(t;u)\leq C_{\lambda,A}(1+t)^\frac{1}{2s}.
\]
\end{theorem}

{\it Remark.}  It is easy to see from this that one also has ${\overline x}_\lambda(t;u)\leq C_{\lambda}t^\frac{1}{2s}$ for $t\ge \tau_{\lambda, A}$, with some $\tau_{\lambda, A}=\tau_{\lambda, A}(s,f)$ but $C_{\lambda}=C_{\lambda}(s,f)$ independent of $A$.

\begin{proof}
Let $k_0$ be the smallest positive integer such that 
${(2s)^{-k}}\geq {\frac{k+1}{1-2s}+1}$ for all $k\ge k_0$. Then we have
\beq\lb{6.22}
\bigcup_{k\geq k_0}\left[2^{\frac{k}{1-2s}},2^{(2s)^{-k}}-2^{\frac{k}{1-2s}}\right]=[2^{\frac{k_0}{1-2s}},\infty).
\eeq
Next let $k_1\geq k_0$ be such that 
\[
2^{\frac{k_1}{1-2s}} \left( 2^{\frac{1}{2s}} \theta_*^{-\frac{1}{2s}} c_* ^{-1} +c_* ^{-\frac{1}{2}} \right)^{-2s}\ge A.
\] 
It follows from \eqref{6.22} that for any $T\geq 2^{\frac{k_1}{1-2s}}$, there is $k\geq k_1$ such that
\[
2^{\frac{k}{1-2s}}\leq T\leq 2^{(2s)^{-k}}-2^{\frac{k}{1-2s}}.
\]
Fix this $T$ and $k$, and let $\Phi^k$ be from \eqref{Phi}.
Then $\Phi^k(2^{\frac{k}{1-2s}},\cdot)\equiv 1$ on $(-\infty,c_* 2^{\frac{k}{(1-2s)2s}}]$, while $\Phi^k$ being non-increasing and $\theta_k^k=\frac{\theta_*}{2}$ show that for $x> c_* 2^{\frac{k}{(1-2s)2s}}$ we have
\[
\Phi^k(2^{\frac{k}{1-2s}},x)\geq \Phi^k(2^{\frac{k}{1-2s}},x+l_k^k(2^{\frac{k}{1-2s}}))\geq \left( 2^{\frac{1}{2s}} \theta_*^{-\frac{1}{2s}}+ c_* ^{-\frac{1}{2}} 2^{-\frac{k}{(1-2s)2s}} x \right)^{-2s}\geq  Ax^{-2s}.
\]
Therefore we have $\Phi^k(2^{\frac{k}{1-2s}},\cdot)\geq u(0,\cdot)$ on $\bbR$.  So if we let $\varphi(t,\cdot):=\Phi^k(t+2^{\frac{k}{1-2s}},\cdot)$, then $\varphi$ is a supersolution to \eqref{1.1} on the time interval $[0,T]$ with $\varphi(0,\cdot)\ge u(0,\cdot)$.  It now follows from the comparison principle
(Theorem \ref{T.2.2}; see also Remark 2 after Definition \ref{D.2.0'})
 that $u\le\phi$ on $[0,T]\times\bbR$.

Since $l_k^k(t)\leq 2t^\frac{1}{2s}$ by \eqref{6.5}, from \eqref{Phi} we obtain for any $\lambda\in (0,\frac{\theta_*}{2})$,
\[
{\overline x}_\lambda(t;\Phi^k)\leq c_* t^\frac{1}{2s}+l_k^k(t)+ \left( \lambda^{-\frac{1}{2s}}- ( \theta_k^k)^{-\frac{1}{2s}} \right) c_* ^\frac{1}{2}t^\frac{1}{2s}  \leq \left(c_* +2 + \sqrt{c_* } \lambda^{-\frac{1}{2s}} \right)t^\frac{1}{2s}
\]
for all $t\ge 1$.  This and $T\ge 2^{\frac{k}{1-2s}}$ show that for all $t\in[0,T]$ we have 
\[
{\overline x}_\lambda(T;u)\leq {\overline x}_\lambda(T+2^{\frac{k}{1-2s}};\Phi^k) \leq \left(c_* +2 + \sqrt{c_* } \lambda^{-\frac{1}{2s}} \right)  2^\frac{1}{2s} T^\frac{1}{2s}.
\]
Since $T\geq 2^{\frac{k_1}{1-2s}}$ was arbitrary, the result now clearly follows for any $\lambda\in(0,1)$, with $C_{\lambda,A}$ depending also on $s,K,\theta_0$ (since $c_* $ and $k_0$ depend on them).
\end{proof}

\section{Subsolutions for Ignition Reactions}\lb{S4}

In this section we prove the lower bound in Theorem \ref{T.1.3}(i), which then automatically  provides the same bound in Theorem \ref{T.1.1}(i) via the comparison principle.  We do so by constructing appropriate subsolutions to \eqref{1.0} in the following  counterpart to Theorem \ref{T.3.3}.  

\begin{theorem}\lb{T.4.3}
Let $f$ and $s$ satisfy (I), and let $0\le u\le 1$ solve \eqref{1.0}.  If
\[
u(0,\cdot) \ge \theta \chi_{B_{R_\theta}(0)}
\]
 for some $\theta>\theta_0$ and $R_\theta$ from Lemma \ref{L.2.4},
then for each $\lambda\in (0,1)$ there are $C_{\lambda}, \tau_{\lambda,\theta}>0$ (depending also on $s,f,d$) such that 
for all $t\ge \tau_{\lambda,\theta}$ we have
\[
{\underline x}_\lambda(t;u)\geq C_{\lambda}\,t^\frac{1}{2s}.
\]
\end{theorem}

\begin{proof}
The comparison principle and Lemma \ref{L.2.4} show that it suffices to prove the result with $C_\lambda$ also depending on $\theta$, which we will do.

Let $u_\theta, \bar u_\theta$ be  from Lemma \ref{L.2.4}, and let $L:=\|u_\theta'\|_\infty<\infty$ and $\supp \,u_\theta=(-\infty,a]$ (so $a\in(0,R_\theta]$). By \eqref{2.18'x}, there is $\eps>0$ such that for all $x\le a$ we have
\beq\lb{4.3'}
(-\Delta)^s \bar u_\theta(x)-f(\bar u_\theta(x))\leq-\eps.
\eeq

Next let
\[
b:= ((2s\eps)^{-1} L a)^{\frac{1}{2s}} \qquad\text{and}\qquad \Psi(t,x):= \bar u_\theta(b t^{-\frac{1}{2s}}x).
\]
Since $\Psi(t,\cdot)\ge 0$ vanishes on $B_{b^{-1}a t^{1/2s}} (0)$ for each $t>0$,
we have
\[
\Psi_t (t,\cdot)+  (-\Delta)^s\Psi(t,\cdot)-f(\Psi(t,\cdot))\leq 0
\]
there.
From \eqref{4.3'} and $f\ge 0$ we obtain for any $t\geq b^{2s}$ and $|x|\leq b^{-1}a t^\frac{1}{2s}$,
\begin{align*}
\Psi_t (t,x)+   & (-\Delta)^s\Psi(t,x)-f(\Psi(t,x))\\
&\leq (2s)^{-1} b t^{-\frac{1}{2s}-1} |x| \, \|u_\theta'\|_\infty + b^{2s}t^{-1} \left[  (-\Delta)^s \bar u_\theta (bt^{-\frac{1}{2s}}x) - f \left( \bar u_\theta(bt^{-\frac{1}{2s}}x) \right) \right]\\
&\leq
\left( (2s)^{-1} L a -b^{2s}\eps \right) t^{-1},
\end{align*}
which is $\leq 0$ by the definition of $b$.
Hence $\Psi$ is a subsolution to \eqref{1.0} on  time interval $(b^{2s},\infty)$.

Since clearly
$
u(0,\cdot)\geq \bar u_\theta=\Psi(b^{2s},\cdot),
$
the comparison principle (Theorem \ref{T.2.2}) yields
\[
u(t,\cdot)\geq \Psi(t+b^{2s},\cdot)
\]
for all $t\geq 0$.
Hence for each $\lambda\in(0,\theta)$ and $t\ge 0$ we have 
\beq\lb{2.21}
{\underline x}_\lambda(t;u)\geq {\underline x}_\lambda(t;\Psi(\cdot+b^{2s},\cdot))= C_{\lambda,\theta} b^{-1}(t+b^{2s})^\frac{1}{2s},
\eeq
where $C_{\lambda,\theta}\in(0,R_\theta)$ is such that $u_\theta(C_{\lambda,\theta})=\lambda$.
The claim now follows for each $\lambda\in (0,\theta)$.

Moreover, it follows from \eqref{2.21} that there are $\tau',C'>0$ such that for all $t\geq\tau'$ we have
\[
\inf_{|x|\leq C't^{1/2s}}u(t,x)\geq \frac{\theta+\theta_0}{2}
\]
Then the last claim in Lemma \ref{L.2.4} shows that for any $\lambda\in [ \theta,1)$ there is $\tau>0$ such that
\[
\inf_{|x|\leq C't^{1/2s}} u(t+\tau,x)\geq \lambda
\]
for all $t\geq\tau'$.
 It follows that for all $t\ge \tau+\tau'$  we have
\[
{\underline x}_\lambda(t;u)\geq C'(t-\tau)^\frac{1}{2s}
\]
(with $C',\tau',\tau$ depending on $s,f,d,\lambda,\theta$), which proves the claim for $\lambda\in [ \theta,1)$.
\end{proof}


\section{Subsolutions for Monostable Reactions} \lb{S5}

In this section we prove Theorems \ref{T.1.1}(ii) and \ref{T.1.3}(ii), so we will assume the following.

\smallskip

\begin{itemize}
    \item[ (M)] Let $f$ be an $\alpha$-monostable reaction for some $\alpha>1$, and let $s\in(0,\frac{\alpha}{2(\alpha-1)})$.
\end{itemize}
\smallskip

  The relevant upper bound on ${\overline x}_\lambda(t;u)$ for front-like data was already obtained in \cite{coville2020propagation} (for $\alpha$-monostable reactions $\tilde f$ that are also $C^1$ and have $\tilde f'(1)<0$, but there always exists such $\tilde f\ge f$, so the same bound for $f$ follows by the comparison principle).  This immediately provides the bounds in Theorems \ref{T.1.1}(ii) and \ref{T.1.3}(ii) via the comparison principle. 
 
 Hence it remains to show that for any $\lambda\in(0,1)$ and all large $t$ we have
\beq\lb{7.11}
{\underline x}_\lambda(t;u)\geq C t^{\frac{\alpha}{2s(\alpha-1)}}.
\eeq
It suffices to do this in the setting of Theorem \ref{T.1.3}(ii), because then the same bound in Theorem~\ref{T.1.1}(ii) follows via the comparison principle. We will do this by constructing appropriate subsolutions to \eqref{1.0}.

Let $\alpha,\theta_0,\gamma$ be from Definition  \ref{D.1.1}(ii) for $f$, and let 
\[
\beta:=(d+2s)(\alpha-1) \qquad\text{and}\qquad \kappa:=
\frac{\beta\alpha}{2s(\alpha-1)} \quad (> \beta)
\]
(with the inequality due to (M)). Fix any $\theta\in (0,1)$, let
\[
\theta_1:=\min\left\{\theta_0, \frac \theta 2 \right\} \qquad\text{and}\qquad \theta_2:=\frac {1+\theta}{2}.
\]
Then let
\beq\lb{8.5tau}
\nu:=\alpha-1 \qquad\text{and}\qquad \tau:=\min\left\{ \tau_0,\,\frac{c_* }{2^{(d+2s)/\beta}C_*\theta_1}\right\},
\eeq
where $\tau_0,c_* , C_*$  are from Lemma \ref{L.8.1} below with $\beta,\nu,\theta_1$ as above 
(so $\tau_0,c_* , C_*$ are independent from $a,b$ in the lemma). 
Next let 
\[
\delta:=\inf_{u\in [\tau\theta_1,\theta]} f(u)>0,
\]
define
\beq\lb{8.5'}
a_3:=\min\left\{ 1,\,\frac{(\alpha-1) \gamma}{2\kappa-1},\,\frac{(\alpha-1) \delta }{4\kappa\theta_2^\alpha}\right\},
\eeq
and then
\beq\lb{8.5}
 a_1:=\left( \frac{(\alpha-1)c_* }{2^{1+(d+2s)/\beta}(2\kappa-1)}\right)^{\frac{\beta}{2s}}{a_3^\frac{d}{2s}}
 \qquad\text{and}\qquad a_2:=a_1a_3.
\eeq

Now let $\phi_\theta:[0,1]\to [0,\theta]$ be smooth and such that
\beq\lb{7.13}
\phi_\theta(y)=y \, \text{ on } [0,\theta_1]\qquad\text{and}\qquad \phi_\theta(y)=\theta \, \text{ on }[\theta_2,1],
\eeq
as well as for some $C_\theta>0$ we have
\beq\lb{7.12}
0\le \phi_\theta'\le 1 \qquad\text{and}\qquad -C_\theta \le \phi_\theta''\le 0.
\eeq
Finally define 
\beq\lb{8.0}
\psi_\theta(t,r):=
   ( a_1^{-1} t^{1-\kappa}(r^{\beta}- a_2t^\kappa))^{-\frac{1}{\alpha-1}}.
\eeq
for $ t>0$ and $r> ( a_2t^\kappa)^{\frac 1\beta}$, and then let
\[
\Psi_\theta(t,x):=
\begin{cases}
    \phi_\theta\left(\psi_\theta(t,|x|)\right)&|x|\geq ( \theta_2^{1-\alpha}a_1t^{\kappa-1}+ a_2t^\kappa)^{\frac 1\beta},\\
    \theta &\text{otherwise. } 
\end{cases}
\]
For any $u\in (0,\theta_2]$ let also
\[
X_t(u):=(u^{1-\alpha}a_1 t^{\kappa-1}+ a_2t^\kappa)^\frac{1}{\beta}
\]
 (hence $\psi_\theta(t,X_t(u))=u$).
This construction shows that $\Psi_\theta$ is a smooth function. We will now show that it is also a subsolution to \eqref{1.0} at all large times.  We note that since we may have $s>\frac 12$ here, this would be difficult if graphs of $\Psi_\theta(t,\cdot)$ (as functions of $x$) had ``concave'' corners, which is the reason for the introduction of the function $\phi_\theta$ above.

We start with a technical lemma, whose proof is easy when $d=1$, but somewhat more involved when $d>1$.  We  postpone the proof to the appendix.

\begin{lemma}\lb{L.8.1}
Let $\theta_1\in (0,1]$ and $\beta>\nu>0$  be such that 
$\frac{\beta}{\nu}\geq d-2$.
Let $0<a\le1\le b$, let $X(u):=( a^{-1}(u^{-{\nu}}+b))^{\frac 1\beta}$ for $u>0$,   and let $\varphi:\bbR^d\to [0,1]$ be smooth and such that
\[
\varphi(x)=   ( a |x|^{\beta}- b)^{-\frac{1}{\nu}}\qquad\text{ when $|x|\geq X(\theta_1)$}
\]
as well as $\varphi(x)\geq \theta_1$ when $|x|\leq X(\theta_1)$.
Then for any $s\in(0,1)$, there are $c_*,C_*>0$ and $\tau_0\in (0,\frac 14]$ (depending only on $s,\beta,\nu,\theta_1,d$) such that for all $|x|\geq X(\tau_0\theta_1)$ we have
\[
(-\Delta)^s\varphi(x)\leq -c_* X(\theta_1)^d |x|^{-d-2s}+ C_*|x|^{-2s} \varphi(x) .
\]
\end{lemma}

We are now ready to construct a localized subsolution for \eqref{1.0} with monostable $f$.

\begin{theorem}\lb{T.sub}
Let $f$ and $s$ satisfy (M).
Then for any $\theta\in (0,1)$, there is $T_\theta\geq 1$ such that
$\Psi_\theta $ above is a subsolution to \eqref{1.0} on the time interval $[T_\theta,\infty)$.
\end{theorem}

\begin{proof}
We drop $\theta$ from $\phi_\theta,\psi_\theta,\Psi_\theta,T_\theta,C_\theta$ in the proof. 
The desired $T$ we obtain here will depend on the various constants in the above setup, and
 we will always assume that $t\geq T \ge 1$. 

Let us start with estimating $| \psi_r|$ from above.
Since $\psi_{rr}>0>\psi_r$,
we see that for all $r\ge X_t(\theta_2)$ we have 

\[
    |\psi_r(t,r)|\leq -\psi_{r}(t,X_t(\theta_2))=\frac{\beta}{\alpha-1}  a_1^{-1}t^{1-\kappa} \theta_2^\alpha X_t(\theta_2)^{\beta-1}
\]
When $\beta<1$, from $\theta_2<1$ and $ a_1\ge  a_2$ we obtain 
\[    
    -\psi_{r}(t,X_t(\theta_2)) \leq \frac{\beta}{\alpha-1}  a_1^{-1}t^{1-\kappa} ( a_2t^\kappa)^\frac{\beta-1}{\beta} 
    \leq  \frac{\beta }{\alpha-1}  a_2^{-\frac{1}{\beta}}t^{\frac{\beta -\kappa}{\beta}}.
\]
When $\beta\ge 1$, we again use  $\theta_2<1$  to get
\[
\theta_2^\alpha  X_t(\theta_2)^{\beta-1} = \theta_2^{\frac{\alpha-1+\beta}{\beta}} ( a_1t^{\kappa-1}+ \theta_2^{\alpha-1}  a_2t^\kappa)^\frac{\beta-1}{\beta} \le ( a_1t^{\kappa-1}+ a_2t^\kappa)^\frac{\beta-1}{\beta}
\le ( a_1t^{\kappa-1})^\frac{\beta-1}{\beta} + ( a_2t^\kappa)^\frac{\beta-1}{\beta},
\]
so then $ a_2\leq  a_1$  and $t\geq 1$ yield
\[
-\psi_{r}(t,X_t(\theta_2))
\leq \frac{\beta}{\alpha-1}\left( a_1^{-\frac{1}{\beta}}t^{\frac{1-\kappa}{\beta}}+ a_1^{-1} a_2^\frac{\beta-1}{\beta}t^{\frac{\beta-\kappa}{\beta}}\right)
\leq  \frac{2\beta }{\alpha-1}   a_2^{-\frac{1}{\beta}}t^{\frac{\beta-\kappa}{\beta}}.
\]

It follows that if $T$ is such that $\frac{2\beta}{\alpha-1}   a_2^{-\frac1{\beta}}T^{\frac{\beta-\kappa}{\beta}}\leq 1$  (recall that $\kappa>\beta$) and $X_T(\theta_2)\geq 1$, then \eqref{7.12} and $\psi_{rr}>0$ yield  for any $e\in\bbS^{d-1}$ and
 $\rho_t:=\frac{2\beta}{\alpha-1}   a_2^{-\frac1{\beta}}t^{\frac{\beta-\kappa}{\beta}}\leq 1$, 
\[
\begin{aligned}
D_{ee}^2\Psi(t,x)&=\frac{(x\cdot e)^2}{|x|^2}\phi'' \big(\psi(t,|x|) \big)|\psi_r(t,|x|)|^2+\frac{(x\cdot e)^2}{|x|^2}\phi' \big(\psi(t,|x|) \big)\psi_{rr}(t,|x|)\\
&\qquad +\phi' \big(\psi(t,|x|) \big)\psi_r(t,|x|) \frac{|x|^2-(x\cdot  e)^2}{|x|^3} \geq-C \rho_t^2-\rho_t\geq -(C+1)\rho_t,
\end{aligned}
\]
where we also used that $\phi' (\psi(t,|x|) ) =0$ when $|x|\le 1$ (due to $X_t(\theta_2)\geq 1$).
From this, $\theta<1$, and $s\in(0,1)$ we obtain (with $\omega_d$ be the surface area of $\bbS^{d-1}$)
\beq\lb{8.1}
\begin{aligned}
  \sup_{x\in\bbR^d}\,  (-\Delta)^s\Psi(t,x)
  &\leq  \frac{c_{s,d}}{2}\int_{|h|\leq \rho_t^{-1/2}} \frac{2\Psi(t,x)-\Psi(t,x+h)-\Psi(t,x-h)}{|h|^{d+2s}}dh\\
  &\qquad\qquad+ c_{s,d} \int_{|h|>\rho_t^{-1/2}} \frac{\Psi(t,x)-\Psi(t,x+h)}{|h|^{d+2s}} dh \\
  &\leq  \frac{c_{s,d}}{2}\int_{|h|\leq \rho_t^{-1/2}} \frac{(C+1) \rho_t}{|h|^{d-2+2s}}dh + c_{s,d} \int_{|h|>\rho_t^{-1/2}} \frac\theta{|h|^{d+2s}} dh\\
  &\leq \frac{c_{s,d}\omega_d (C+1)\rho_t^{s}}{2s(1-s)}=:C'\rho_t^s.
\end{aligned}
\eeq

We will need another estimate for $(t,x)$ such that $\Psi(t,x)$ is small.  When $\Psi(t,x)\leq \tau\theta_1$ (then $\Psi(t,x)=\psi(t,x)$ because $\tau\leq \tau_0\le 1$), we can apply Lemma \ref{L.8.1} with 
$\varphi(\cdot)= \Psi(t,\cdot)$ and $\beta,\nu,\theta_1$  as above,
provided $T$ is large enough so that $a:= a_1^{-1}t^{1-\kappa}<1$ and $b:=a_2a_1^{-1}t\ge 1$.  
So when $\Psi(t,x)\leq \tau\theta_1$, then we have
\beq\lb{8.3}
(-\Delta)^s \Psi(t,x)\leq -c_* |x|^{-d-2s}X_t(\theta_1)^d+ C_* |x|^{-2s} \psi(t,x).
\eeq

Finally, we also note that $\Psi_t(t,x)=0$ when $|x|\leq X_t(\theta_2)$, while for $|x|>X_t(\theta_2)$ we have
\beq\lb{8.2}
\begin{aligned}
    \Psi_t(t,x)&= \phi'(\psi(t,x))\frac{\psi(t,x)^\alpha  a_1^{-1}}{\alpha-1} \left((\kappa-1)t^{-\kappa}|x|^{\beta}+  a_2\right)\\
    &\leq \frac{\psi(t,x)^\alpha  a_1^{-1}}{\alpha-1} \left((\kappa-1)  a_1t^{-1} \psi(t,x)^{1-\alpha}+\kappa  a_2 \right) \\
    & = \frac{\kappa-1}{\alpha-1} \, \frac{\psi(t,x)}t+\frac{ a_2 \kappa }{ a_1(\alpha-1)} \, \psi(t,x)^{\alpha} .
\end{aligned}
\eeq

We are now ready to show that $\Psi_t+(-\Delta)^s\Psi-f(\Psi)\leq 0$ at all $(t,x)$ with $t\ge T$ (if $T$ is large enough). When $|x|\leq X_t(\theta_2)$, then \eqref{8.1} and  $f(\Psi(t,x))=f(\theta)\ge\delta$  show that this follows from $C' \rho_t^s\leq {\delta}$,
which holds if $T$ is large.
Since  \eqref{8.5} and \eqref{8.5'} also yield
\[
\max \left\{\frac{\kappa-1}{\alpha-1} \frac{\theta_2}{t} ,\, \frac{ a_2\kappa }{ a_1(\alpha-1)} \theta_2^\alpha ,\, C' \rho_t^s \right\}\le\frac\delta 3
\]
if $T$ is large,
and since for $|x|\in (X_t(\theta_2), X_t(\tau\theta_1)]$ we still have $f(\Psi(t,x))\ge\delta$, it follows from \eqref{8.1} and \eqref{8.2}  that for these $x$ we again have
\[
\Psi_t(t,x)+ (-\Delta)^s\Psi (t,x) -f(\Psi(t,x))\leq \frac{\kappa-1}{\alpha-1}\frac{\theta_2}{t} + \frac{ a_2\kappa }{ a_1(\alpha-1)} \theta_2^\alpha +C' \rho_t^{s} 
-\delta\leq 0.
\]

It therefore remains to consider $|x|>X_t(\tau\theta_1)$ (when $\Psi(t,x)= \psi(t,x)$ because $\tau\le 1$). 
If also $t \Psi(t,x)^{\alpha-1}\leq \frac { a_1}{ a_2}$, 
it follows from \eqref{8.3}, \eqref{8.2}, $|x|=(\psi(t,x)^{1-\alpha} a_1t^{\kappa-1}+ a_2t^\kappa)^{-\frac{1}{\beta}}$, and $X_t(\theta_1)\geq (a_2t^\kappa)^\frac{1}{\beta}$ that (we drop $(t,x)$ from the notation for simplicity)
\begin{align*}
    \Psi_t+ (-\Delta)^s\Psi-f(\Psi)&\leq   \frac{\kappa-1}{\alpha-1}\frac\psi t+\frac{ a_2\kappa }{ a_1(\alpha-1)} \psi^{\alpha}-c_*  |x|^{-d-2s}X_t(\theta_1)^d +C_* |x|^{-2s} \psi\\
    &\leq   \frac{2\kappa-1}{\alpha-1}\frac\psi t-c_* (2\psi^{1-\alpha} a_1t^{\kappa-1})^{-\frac{d+2s}{\beta}}(a_2t^\kappa)^\frac{d}{\beta} +C_*(\psi^{1-\alpha} a_1t^{\kappa-1})^{-\frac{2s}{\beta}}  \psi \\
    &= \frac{2\kappa-1}{\alpha-1} \frac\psi t-c_* (2 a_1)^{-\frac{d+2s}{\beta}}a_2^\frac{d}{\beta}\frac{\psi}{t}+C_* a_1^{-\frac{2s}{\beta}} t^{-\frac{(\kappa-1)2s}{\beta}} \psi^{1+\frac{(\alpha-1)2s}{\beta}}.
\end{align*}
Using again $t\psi^{\alpha-1}\leq \frac{a_1}{a_2}$, we obtain
\[
t^{-\frac{(\kappa-1)2s}{\beta}}  \psi^{1+\frac{(\alpha-1)2s}{\beta}}
= t^{\frac{2s}{\beta}}  \psi^{\frac{(\alpha-1)2s}{\beta}} t^{-\frac\alpha{\alpha-1}} \psi
\leq (a_1a_2^{-1})^\frac{2s}{\beta} t^{-\frac{1}{\alpha-1}}\frac{\psi}{t}.
\]
Therefore
$\Psi_t+ (-\Delta)^s\Psi-f(\Psi)\leq 0$ by \eqref{8.5} if $T$ is large enough.

When instead $|x|>X_t(\tau\theta_1)$ and $t \Psi(t,x)^{\alpha-1}> \frac { a_1}{ a_2}$, then 
\[
(a_2t^\kappa)^\frac{1}{\beta} \le |x|=(\psi(t,x)^{1-\alpha} a_1t^{\kappa-1}+ a_2t^\kappa)^{-\frac{1}{\beta}} \le  (2a_2t^\kappa)^\frac{1}{\beta},
\]
so \eqref{8.3}, $X_t(\theta_1)\geq (a_2t^\kappa)^\frac{1}{\beta}$, $\psi(t,x) \leq \tau\theta_1$ (due to $|x|>X_t(\tau\theta_1)$), and  \eqref{8.5tau} show that
\begin{align*}
(-\Delta)^s\Psi \leq -c_*   (2a_2t^\kappa)^{-\frac{d+2s}{\beta}}(a_2t^\kappa)^\frac{d}{\beta}+ C_* (a_2t^\kappa)^{-\frac{2s}{\beta}} \psi 
\leq  \left( C_* \tau \theta_1 -c_*  2^{-\frac{d+2s}{\beta}} \right) a_2^{-\frac{2s}{\beta}}t^{-\frac{2s\kappa}{\beta}} \le 0.
\end{align*}
This, \eqref{8.2}, $\frac{\psi(t,x)}{t}<\frac{a_2}{a_1}\psi(t,x)^\alpha$,  and  (M) show that
\begin{align*}
    \Psi_t+ (-\Delta)^s\Psi-f(\Psi)&\leq
\frac{ a_2(2\kappa-1)}{ a_1(\alpha-1)} \psi^{\alpha}-\gamma \psi^\alpha
\end{align*}
which is again $\le 0$ due to \eqref{8.5'}.
This finishes the proof.
\end{proof}

We can now use the constructed subsolutions to prove \eqref{7.11}.

\begin{theorem}\lb{T.7.1}
Let $f$ and $s$ satisfy (M), and let $0\le u\le 1$ solve \eqref{1.0}. If
\[
u(0,\cdot) \ge \theta \chi_{B_{R_\theta}(0)}
\]
 for some $\theta> 0$ and $R_\theta$ from Lemma \ref{L.2.4},
then for each $\lambda\in (0,1)$ there are $C_{\lambda},\tau_{\lambda,\theta}>0$ (depending also on $ s,f,d$) such that for all $t\ge \tau_{\lambda,\theta}$ we have
\[
{\underline x}_\lambda(t;u)\geq C_{\lambda}t^{\frac{\alpha}{2s(\alpha-1)}}.
\]
\end{theorem}

\begin{proof}
The comparison principle and Lemma \ref{L.2.4} show that it suffices to prove the result with $C_\lambda$ also depending on $\theta$, which we will do.

Let $\bar u_\theta$ be from Lemma \ref{L.2.4} (see the remark after that lemma), let $\bar u$ be the solution to \eqref{1.0} with initial data $\bar u_\theta$, and let $t_0$ be such that $\bar u(t_0,\cdot)\ge \theta \chi_{B_1(0)}$.  Since $u_0\ge \bar u_\theta$, we have $u(t_0,\cdot)\ge\bar u(t_0,\cdot)$ by the  comparison principle (Theorem \ref{T.2.2}), and then comparison principle shows that it suffices to consider $u_0=\bar u(t_0,\cdot)$ without loss.  The proof of Lemma \ref{L.2.4} now shows that $u$ is time increasing.

Similarly to \cite[Theorem 3.1]{coville2020propagation}, since $u$ dominates the solution to $v_t+(-\partial_{xx})^s v=0$ with initial data $\theta \chi_{B_{1}(0)}$, there is $C>0$ (depending only on $s,d$) such that if $t\ge 1 $ and $|x|\ge t^\frac{1}{2s}+1$, then
\begin{align*}
    u(t,x)\geq C\theta\int_{B_{1}(x)} t^{-\frac{d}{2s}}(1+|t^{-\frac{1}{2s}}y|^{d+2s})^{-1}dy
    \geq \frac{C\theta \omega_d}{2d}\, t(|x|+1)^{-d-2s} \ge c\, t|x|^{-d-2s},
\end{align*}
where $\frac{\omega_d}d$ is the volume of $B_1(0)$ and $c:= 2^{-d-2s-1}d^{-1}C\theta \omega_d$.

Now let $\Psi_\theta,T_\theta$ be from Theorem \ref{T.sub} and let ${T}:=1+ c^{-1}a_1^{\frac 1{\alpha-1}}T_\theta^\frac{\kappa-1}{\alpha-1}$.  If $|x|$ is  large enough, we then have ${u}({T},x)\geq\Psi_\theta(T_\theta,x)$, which then yields ${u}({t},x)\geq\Psi_\theta(T_\theta,x)$ for all these $x$ and all $t\ge T$ because $u$ increases in time.  But we also have ${u}({t},x)\geq\Psi_\theta(T_\theta,x)$ for all the other $x$ and  some $t$ by the last claim in Lemma \ref{L.2.4}. 
Hence there is ${T}'\geq {T}$ such that ${u}({T}',\cdot)\geq \Psi_\theta(T_\theta,\cdot)$.  Comparison principle now yields
\[
{u}(t+{T'},\cdot)\geq \Psi_\theta(t+T_\theta,\cdot)
\]
for all $t\ge 0$,
 so for any $\lambda\in (0,\theta)$ and $t\geq {T'}$ we have
\[
{\underline x}_\lambda(t;u)
\geq {\underline x}_\lambda(t-{T'}+T_\theta;\Psi_\theta)\geq C_{\lambda,\theta} (t-{T'}+T_\theta)^{\frac{\alpha}{2s(\alpha-1)}}
\]
for some  time-independent $C_{\lambda,\theta}>0$.  This proves the claim for each $\lambda\in (0,\theta)$, and for $\lambda\in [\theta,1)$ it now follows as at the end of the proof of Theorem \ref{T.4.3}.
\end{proof}

\appendix

\section{Proof of Lemma \ref{L.2.4}}

We will first show that it suffices to obtain existence of $R_\theta,u_\theta$ satisfying \eqref{2.20} and \eqref{2.18'} (note that \eqref{2.18''} then follows from $u_\theta=0$ on $(R_\theta,\infty)$).
Let us assume this is the case, and for any $R\geq 0$ and all $x\in\bbR$ let $v_R(x):= u_\theta(|x|-2R)$.
Then $v_R\equiv 0$ on $B_{2R+R_\theta}(0)^c$, so on this set we have  $-(-\Delta)^s v_R\geq 0$.
Since $v_R=u_\theta(|\cdot|-2R)$ on $B_R(x)$ when $|x|\in[R,2R+R_\theta]$ (so as $R\to\infty$, uniformly in these $x$ we have local uniform (in $y$) convergence of $v_R(y+x)$ to $u_\theta(y\cdot \frac x{|x|} +|x| -2R)$ in $C^2$), and $v_R\equiv \theta$ on $B_R(x)$ when $|x|\le R$, the strict inequality in \eqref{2.18'} and $f(\theta)>0$ guarantee that for any large enough $R$ we have
\[
\inf_{|x|\le 2R+R_\theta} \left[-(-\Delta)^s v_R(x)+f(v_R(x)) \right] > 0.
\]
(Recall also that $c_{s,d}=c_{s,1}(\int_{\bbR^{d-1}} (1+h^2)^{-\frac d2 -s} dh)^{-1}$.)
By symmetry this also holds with $|x|\le R_\theta+R$ under the $\inf$, so 
\eqref{2.18''x} and  \eqref{2.18'x} hold for $\bar u_\theta:=v_R$ when $R$ is large enough (and we then replace $R_\theta,u_\theta$ by $2R+R_\theta,u_\theta(\cdot-2R)$). 
%

Hence to prove the first claim, it remains to  find $R_\theta,u_\theta$ satisfying \eqref{2.20} and \eqref{2.18'}.
Let us now assume that there is  Lipschitz continuous,  piecewise smooth (and linear on both sides of each point where it is not smooth), non-increasing $\varphi:\bbR\to [0,\theta]$ and $R'>0$ such that 
\begin{enumerate}
   \item $\varphi=\theta$ on $(-\infty,0]$ and $\varphi=0$ on $[R',\infty)$;
    
    \smallskip
    
    \item $-(-\partial_{xx})^s\varphi> 0$ on the set $\{x\in\bbR\,|\, \varphi(x)\leq \theta_0'\}$, where $\theta_0':= \frac{3\theta_0+\theta}{4} \, (\in(\theta_0,\theta))$;
    
     \smallskip
    
    \item $C:=\sup_{x\in\bbR} \,(-\partial_{xx})^s\varphi(x) <\infty$.
\end{enumerate}
Here $-(-\partial_{xx})^s\varphi$ is allowed to be $\infty$ at the (finitely many) points where $\varphi$ is not smooth  (when $s\geq\frac{1}{2}$).
If $R_\theta:=rR'$ and $u_\theta(x):=\varphi(\frac xr)$ for some $r>0$, then for any $x$ such that $u_\theta(x)\leq \theta_0'$ we have
\[
-(-\partial_{xx})^s u_\theta(x)+f(u_\theta(x))\geq -(-\partial_{xx})^s u_\theta(x)= - r^{-2s} (-\partial_{xx})^s \varphi(xr^{-1})> 0.
\]
If we let $\delta:=\inf_{u\in [\theta_0',\theta]} f(u)>0$ and $r:=(2C/\delta)^{\frac1{2s}}$ (with $C$ from (3)), then for any $x$ such that $u_\theta(x)\geq \theta_0'$ we have
\[
-(-\partial_{xx})^s u_\theta(x)+f(u_\theta(x))\geq -Cr^{-2s}+\delta>0.
\]
Continuity of the left-hand side in $x$ (as a function with values in $\bbR\cup\{\infty\}$) now yields \eqref{2.18'}, and \eqref{2.20} is obvious. Finally a mollification of $u_\theta$ provides the desired smooth function thanks to the sharp inequality in \eqref{2.18'}.

So the it remains to construct $\varphi$. 
Consider a smooth non-decreasing $\psi:\bbR \to \bbR$ such that 
\beq\lb{7.12'}
\psi(y)=y \, \text{ on }\left(-\infty,\frac{\theta+\theta_0'}{2}\right] \qquad\text{and}\qquad \psi(y)=\theta \, \text{ on } \left[ \theta ,\infty \right)
\eeq
(it will play the same role as $\phi_\theta$ in Section \ref{S5}, preventing concave corners on the graph of $\varphi$).  Let $N\geq 1$ be the smallest integer such that $\theta-\theta_0'\geq 2^{-N}  \theta$, and let us first assume that $N=1$. Set
\[
 l_0(x):=\theta-(\theta-\theta_0')x,  \qquad  k_1:=\frac{\theta-\theta_0'}{2}, \qquad b_1:=\frac{\theta+\theta_0'}{2},  \qquad l_1(x):=b_1-k_1 x,
\]
and define $\varphi_1:\bbR\to [0,1]$ via
\[
\varphi_1(x):=\left\{
\begin{aligned}
&\max\{\psi(l_0(x)), l_1(x), 0\} &&\quad \text{ for }x\geq 0,\\
&\theta &&\quad  \text{ for }x\leq 0.
\end{aligned}\right.
\]
Then $\varphi_1$ is clearly Lipschitz continuous and non-increasing, and from $l_1< \psi\circ l_0$ on $(-1,1)$ (note that $l_1< \theta=\psi\circ l_0$ on $(-1,0]$, while $l_1< \min\{l_0, \frac{\theta+\theta_0'}{2}\}\le \psi\circ l_0$ on $(0,1)$) we have
\[
\varphi_1=\psi\circ l_0 \quad \text{on $\left [-1, 1 \right]$}, \qquad \varphi_1(1)=\theta_0',  \qquad   \varphi_1=l_1 \quad \text{on $\left[1, \frac{b_1}{k_1} \right]$}, \qquad \varphi_1 \left(\frac{b_1}{k_1} \right)=0.
\]
Since $\varphi_1$ is convex on $[\frac 12,\infty)$, and $\psi$ is smooth and satisfies \eqref{7.12'}, we have $\sup_{x\in\bbR} (\varphi_1)_{xx}(x)>-\infty$.  Hence a computation similar to \eqref{8.1} proves (3) for $\varphi_1$.

From $N=1$ we see that $\theta_0'\leq \frac \theta 2$, and so $l_1(1)\le \frac\theta 2$ and $\frac{b_1}{k_1}\le 3$.
Hence for any $x\in [1, \frac{b_1}{k_1}]$ we have   $2x+1\ge \frac{b_1}{k_1}$ and $l_1(x)\le \frac\theta 2$, which together with
$\varphi_1> l_1$ on $(-1,1)$ and $l_1\ge\theta$ on $(-\infty,-1]$ yields
\begin{align*}
-(-\partial_{xx})^s\varphi_1(x)&=c_s \int_{0}^\infty \frac{\varphi_1(x+h)+\varphi_1(x-h)-2\varphi_1(x)}{h^{1+2s}}dh\\
&>    c_s \int_{0}^{x+1}\frac{l_1(x+h)+l_1(x-h)-2l_1(x)}{h^{1+2s}}dh+ c_s \int_{x+1}^{\infty}\frac{\theta-2l_1(x)}{h^{1+2s}}dh 
\geq 0.
\end{align*}
For $x\geq \frac{b_1}{k_1}$, we obviously  have $-(-\partial_{xx})^s\varphi_1(x)>0$ because $\varphi_1(x)=0\le\varphi_1 $. Therefore $-(-\partial_{xx})^s\varphi_1>0$ on $[1,\infty)$, hence (1)--(3) follows with $\varphi:=\varphi_1$ and $R':=\frac{b_1}{k_1}$.


Next assume that $N\ge 2$, and let $\psi,k_1,b_1,l_0,l_1$ be as above.  Since now $\theta_0'-(2\theta_0'-\theta)\le \frac{\theta-(2\theta_0'-\theta)}2$ (in fact, equality holds here), the above argument applies to the function
\[
\tilde \varphi_1(x):=\left\{
\begin{aligned}
& \max\{\psi(l_0(x)), l_1(x), 2\theta_0'-\theta\} &&\quad \text{ for }x\geq 0,\\
&\theta &&\quad  \text{ for }x\leq 0,
\end{aligned}\right.
\]
which is equal to $\varphi_1$ above  on $(-\infty,3]$ and to $2\theta_0'-\theta$ on $[3,\infty)$ (because $l_1(3)=2\theta_0'-\theta>0$). Hence $-(-\partial_{xx})^s\tilde{\varphi}_1>0$ on $[1,\infty)$.
We will now change $\varphi_1$ to $l_2(x):=b_2-k_2x$ on  $[x_2,\frac{b_2}{k_2}]$, where
\[
x_2:=3,
 \qquad b_2:=k_2 x_2+2\theta_0'-\theta,
\]
and $k_2\in(0,k_1)$ is to be determined (notice that $l_2(x_2)=2\theta_0'-\theta = l_1(x_2)$, and hence $k_2<k_1$ shows that $b_2=l_2(0)<l_1(0)<\theta$).  So we let
\[
\varphi_2(x):=\left\{
\begin{aligned}
& \max\{\psi(l_0(x)), l_1(x), l_2(x), 0\} &&\quad \text{ for }x\geq 0,\\
&\theta &&\quad  \text{ for }x\leq 0.
\end{aligned}\right.
\]
Since $\varphi_2 \to\tilde{\varphi}_1 $ locally uniformly on $\bbR$ as $k_2\to 0$, there is $k_2\in(0,k_1)$ such that $-(-\partial_{xx})^s{\varphi_2}> 0$ on $[1,x_2]$.  Fix one such $k_2$ and the corresponding $\varphi_2$ (which again satisfies (3) as above).

If now $N=2$,  consider any $x\in  [x_2,\frac{b_2}{k_2}]$.  
From $l_2(\frac{b_2-\theta}{k_2})=\theta$ 
and $l_2(x_2)=2\theta_0'-\theta\le \frac\theta 2$ we see that $l_2<\varphi_2$ on $(\frac{b_2-\theta}{k_2},x_2)$, and  $\frac{b_2}{k_2}\le 2x_2+\frac{\theta-b_2}{k_2}$.  Hence for any $x\in [x_2, \frac{b_2}{k_2}]$ we have  $2x+\frac{\theta-b_2}{k_2}\ge \frac{b_2}{k_2}$ and $l_2(x)\le \frac\theta 2$, and so
\begin{align*}
-(-\partial_{xx})^s\varphi_2(x)&=c_s \int_{0}^\infty\frac{\varphi_2(x+h)+\varphi_2(x-h)-2\varphi_2(x)}{h^{1+2s}}dh\\
&>    c_s \int_{0}^{x+\frac{\theta-b_2}{k_2}}\frac{l_2(x+h)+l_2(x-h)-2l_2(x)}{h^{1+2s}}dh+ c_s \int_{x+\frac{\theta-b_2}{k_2}}^{\infty}\frac{\theta-2l_2(x)}{h^{1+2s}}dh
\geq 0.
\end{align*}
(Note that this is the same argument as for $N=1$, but with $-1$ and 1 replaced by $\frac{b_2-\theta}{k_2}$ and $x_2$.)
For $x\geq \frac{b_2}{k_2}$, we again have $-(-\partial_{xx})^s\varphi_2(x)>0$ because $\varphi_2(x)=0\le\varphi_2$. Therefore $-(-\partial_{xx})^s\varphi_2>0$ on $[1,\infty)$, which yields (1)--(3) with $\varphi:=\varphi_2$ and $R':=\frac{b_2}{k_2}$.

If $N\ge 3$, the above argument  instead applies to 
\[
\tilde \varphi_2(x):=\left\{
\begin{aligned}
& \max\{\psi(l_0(x)), l_1(x), l_2(x), 4\theta_0'-3\theta\} &&\quad \text{ for }x\geq 0,\\
&\theta &&\quad  \text{ for }x\leq 0,
\end{aligned}\right.
\]
which is equal to $\varphi_2$ on $(-\infty,2x_2+\frac{\theta-b_2}{k_2}]$ and to $4\theta_0'-3\theta$ on $[2x_2+\frac{\theta-b_2}{k_2},\infty)$ (because now $l_2(2x_2+\frac{\theta-b_2}{k_2})=2(2\theta_0'-\theta)-\theta=4\theta_0'-3\theta>0$). Hence again $-(-\partial_{xx})^s\tilde{\varphi}_2>0$ on $[1,\infty)$.  Similarly to the case $N\ge 2$, we let 
\[
x_3:=2x_2+\frac{\theta-b_2}{k_2},\qquad b_3:=k_3 x_3+4\theta_0'-3\theta, \qquad l_3(x):=b_3-k_3x,
\]
with $k_3\in(0,k_2)$ small enough so that
\[
\varphi_3(x):=\left\{
\begin{aligned}
& \max\{\psi(l_0(x)), l_1(x), l_2(x), l_3(x), 0\} &&\quad \text{ for }x\geq 0,\\
&\theta &&\quad  \text{ for }x\leq 0.
\end{aligned}\right.
\]
satisfies $-(-\partial_{xx})^s{\varphi_3}> 0$ on $[1,x_3]$.  If $N=3$, we can use $l_3(x_3)=4\theta_0'-3\theta\le \frac\theta 2$ and $k_3<k_2$ to again show as above that (1)--(3) hold with $\varphi:=\varphi_3$ and $R':=\frac{b_3}{k_3}$.

If $N\ge 4$, this argument can be repeated finitely many times until we obtain a function $\varphi_N$ and $b_N,k_N>0$ such that (1)--(3) hold with $\varphi:=\varphi_N$ and $R':=\frac{b_N}{k_N}$.
\medskip

Finally, let us prove the last claim.  Without loss, we can assume that $x_0=0$; the comparison principle (Theorem \ref{T.2.2}) then shows that it suffices to consider $u(0,\cdot)=\bar u_\theta$.  
We now have $ u(t,\cdot)\geq \bar u_\theta= u(0,\cdot)$ for all $t\geq 0$ by \eqref{2.18} and the comparison principle, 
so applying the comparison principle to $u$ and its time shifts now shows that $u$ is non-decreasing in time. If we let  $v(x):=\lim_{t\to \infty} u(t,x)\le 1$,  Theorem \ref{T.2.1} implies that $v\in C^{2s+\sigma}(\bbR^d)$ for some $\sigma>0$, and $-(-\Delta)^sv+f(v)= 0$ holds in the classical sense.  Since $v\ge \bar u_\theta$, \eqref{2.18'x} shows that $v>\bar u_\theta$ on $B_{R_\theta}(0)$.  But then $u(\tau,\cdot)\ge \sup_{|y|\le r} \bar u_\theta (\cdot-y)$ for some $\tau,r>0$.  By iterating this argument we obtain $u(n\tau,\cdot)\ge \sup_{|y|\le nr} \bar u_\theta (\cdot-y)$ for all $n\in\bbN$, so $v\ge\theta$.  Since $f>0$ on $[\theta,1)$, it is easy to show that the only stationary classical solution to \eqref{1.0} taking values in $[\theta,1]$ is $v\equiv 1$ (note that Theorem \ref{T.2.1} shows that all such solutions are uniformly bounded in $C^{2s+\sigma}(\bbR^d)$), and the claim follows.

\section{Proof of Lemma \ref{L.8.1}}

Let us fix any $x\in\bbR^d$ such that $|x|\geq X(\tau_0\theta_1)$, with $\tau_0\in(0,\frac 14]$ to be determined. Then
\begin{align*}
c_{s,d}^{-1} (-\Delta)^s\varphi(x)&\leq  \int_{|h|\leq |x|-X(\theta_1)} \frac{\varphi(x)-\varphi(x+h)}{|h|^{d+2s}}dh+ \int_{|x|-X(\theta_1)\le |h|\leq |x|} \frac{{\varphi}(x)-\tilde{\varphi}(x+h)}{|h|^{d+2s}}dh\\
& + \int_{ |h|\le |x|  \,\&\, |x+h|\leq X(\theta_1)}\frac{\varphi(x)-2^{-1}\theta_1}{|h|^{d+2s}}dh+ \int_{|h|\geq |x|}\frac{\varphi(x)}{|h|^{d+2s}}dh=: I_1+I_2+I_3+I_4,
\end{align*}
where $I_1$ is a principal value integral and
\[
\tilde{\varphi}(\cdot):=
\varphi(\cdot)-2^{-1}\theta_1\chi_{B_{X(\theta_1)}(0)}(\cdot) \ge 2^{-1}\theta_1\chi_{B_{X(\theta_1)}(0)}(\cdot).
\]
Since $\varphi(x)\leq \tau_0\theta_1\leq \frac{1}{4}\theta_1$ and $|x|\geq X(\theta_1)$, there is $\mu_d>0$ (only depending on $d$) such that
\begin{align*}
I_3        \leq -\int_{|h|\le |x|  \,\&\, |x+h|\leq X(\theta_1)} \frac{\theta_1}{4|h|^{d+2s}}dh 
\leq -\mu_d X(\theta_1)^d  \theta_1 |x|^{-d-2s}.
\end{align*}
We now let $c_*:=c_{s,d}\mu_d\theta_1$, which means that it remains to show that
\[
I_1+I_2+I_4\le c_{s,d}^{-1} C_* |x|^{-2s} \varphi(x),
\] 
with $C_*$ to be determined.

If now $g(l):=(al^\beta-b)^{-\frac{1}{\nu}}$  for $l> (a^{-1}b)^{\frac{1}{\beta}}$ (then $g(|y|):=\varphi(y)$ for $|y|\geq X(\theta_1)>(a^{-1}b)^{\frac{1}{\beta}}$), then using  $g(l)^{-\nu}\geq al^\beta $ yields 
\begin{align*}
g''(l)+\frac{d-1}{l}g'(l)    &=\nu^{-2}(1+\nu)g(l)^{1+2\nu}a^2\beta^2 l^{2\beta-2}-{\nu^{-1}}g(l)^{1+\nu}a\beta(\beta+d-2) l^{\beta-2} \\
    &\geq \nu^{-2}g(l)^{1+2\nu}a^2\beta l^{2\beta-2}(\beta-\nu(d-2)).
\end{align*}
This is $\geq 0$ due to $\frac{\beta}{\nu}\geq d-2$,
so $\varphi$ is subharmonic on $\left(B_{X(\theta_1)} (0)\right)^c\supseteq B_{|x|-X(\theta_1)}(x)$. Hence for any $r\in(0,|x|-X(\theta_1))$ we have  $\dashint_{\partial B_r(x)}\varphi(y)d\sigma(y) \geq\varphi(x)$, and so $I_1\leq 0$.  We also have
\begin{align*}
    I_4= \varphi(x)\int_{|h|\geq |x|} |h|^{d+2s}dh\leq \mu_d' |x|^{-2s} \varphi(x)
\end{align*}
for some $\mu_d'>0$ only depending on $d$.  It therefore remains to estimate $I_2$.

When $d=1$, we  get $I_2\le 0$ because ${\varphi}(x)-\tilde{\varphi}(x+h)$ is no more than $\frac{\theta_1} 4- \frac{\theta_1} 2=-\frac{\theta_1} 4$ for $h\in[0,X(\theta_1)]$ and no more than $\frac{\theta_1} 4$ for $h\in [2x-X(\theta_1), 2x]$ (this is when $x>0$; when $x<0$, these two intervals must be reflected across 0).  This finishes the proof when $d=1$.

We will need to work a little harder when $d\ge 2$.
Let  $\eps_1:=|x|^{-1}X(\theta_1)$ (which is $<1$ because $|x|\geq X(\tau_0\theta_1)$)
and $\tau:=\frac{\varphi(x)}{\theta_1} \le\tau_0$. 
Let us first consider the case when $\eps_1>\frac23$, so that $\rho_1:=1-\eps_1 < \frac13$ (and then $|x|-X(\theta_1)=\rho_1|x|$). 
There is $c_d>0$ such that for all $r\in[ 2\rho_1|x|, |x|]$ we have, 
\[
\calH^{d-1}\left( \left\{h\,\big|\,|h|=r \,\&\, |h+x|\leq X(\theta_1) \right\} \right)\geq c_d\calH^{d-1} \left( \left\{ h\,\big|\, |h|=r  \right\}\right),
\]
with $\calH^{d-1}$ the $(d-1)$-dimensional measure.
If $\tau_0\leq \frac{c_d}4$, then this, $\tilde{\varphi}\geq \frac{\theta_1}2$ on $B_{X(\theta_1)}(0)$, and $\varphi(x)\le\tau_0\theta_1\le\frac{\theta_1}4$ yield (with $\omega_d$ the surface area of $\bbS^{d-1}$)
\begin{align*}
\int_{|h|=r} [\varphi(x)-\tilde{\varphi}(x+h)]d\sigma(h) & \leq    \int_{|h|=r} \tau_0\theta_1 d\sigma(h)-    \int_{|h|=r\& |h+x|\leq X(\theta_1)} 2^{-1}\theta_1d\sigma(h)\le 0 \\
&\leq  -4^{-1}c_d\omega_d \theta_1 r^{d-1}. 
\end{align*}
From this we obtain (with $\omega_d$ the surface area of $\bbS^{d-1}$)
\begin{align*}
I_2&\leq \int_{\rho_1|x|\leq|h|\leq 3\rho_1|x|}  \frac{{\varphi}(x)-\tilde{\varphi}(x+h)}{|h|^{d+2s}}  dh\\
&\leq -    \int_{2\rho_1|x|}^{3\rho_1|x|} \frac{c_d\omega_d \theta_1r^{d-1}}{4r^{d+2s}}dr+    \int_{\rho_1|x|\leq |h|\leq 2\rho_1|x|} \frac{\varphi(x)}{|h|^{d+2s}}dh\\
&\leq -  \frac{c_{d}\omega_d\theta_1}{8s}(2^{-2s}-3^{-2s})(\rho_1|x|)^{-2s}+\frac{\tau_0\omega_d\theta_1}{2s}(1-2^{-2s})(\rho_1|x|)^{-2s}
\end{align*}
which is $\leq 0$ provided $\tau_0\leq \frac{c_d(4^{-s}-9^{-s})}{4(1-4^{-s})}$.

Finally, we are left with the case $\eps_1\leq \frac23$ (and so $\rho_1\geq \frac13$).
Let
\[
A_x:= \left\{ h\,\big|\, |x|-X_t(\theta_1)\le  |h|\leq |x| \,\&\,|x+h|\geq X(\theta_1) \right\},
\]
and let
\[
e:=|x|^{-1}x \qquad\text{and}\qquad  \eps_2:= |x|^{-1} (a^{-1}b)^{\frac{1}{\beta}}.
\]
Then $\eps_1^\beta-\eps_2^\beta= |x|^{-\beta} a^{-1}\theta_1^{-\nu}\geq 0$.
By again using that $\tilde{\varphi}\ge \frac {\theta_1}2\geq \varphi(x)$ on $B_{ X(\theta_1)}(0)$ and then changing variables via $h=|x|z$, we obtain
\begin{align*}
I_2&\leq\int_{A_x} \frac{\varphi(x)-\varphi(x+h)}{|h|^{d+2s}}dh\\
&={|x|^{-2s}} \int_{\rho_1\leq |z|\leq 1\,\&\, |e+z|\leq \eps_1} \frac{(a|x|^{\beta}-b)^{-\frac{1}{\nu}}-(a|e+z|^\beta |x|^{\beta} -b)^{-\frac{1}{\nu}}}{|z|^{d+2s}}dz\\
&=|x|^{-2s}  \varphi(x) \int_{\rho_1\leq |z|\leq 1\,\&\, |e+z|\leq \eps_1} \frac{(|e+z|^\beta-\eps_2^\beta)^{\frac{1}{\nu}}-(1-\eps_2^\beta)^{\frac{1}{\nu}}}{(|e+z|^\beta-\eps_2^\beta)^{\frac{1}{\nu}}|z|^{d+2s}}dz.
\end{align*}
So it remains to show that
\[
I_2':=\int_{\rho_1\leq |z|\leq 1\,\&\, |e+z|\leq \eps_1} \frac{(|e+z|^\beta-\eps_2^\beta)^{\frac{1}{\nu}}-(1-\eps_2^\beta)^{\frac{1}{\nu}}}{(|e+z|^\beta-\eps_2^\beta)^{\frac{1}{\nu}}|z|^{d+2s}}dz
\]
is uniformly bounded  above for $\rho_1\in [\frac13,1]$, $0\le \eps_2\leq \eps_1\leq\frac23$, and $e\in\bbS^{d-1}$, by a constant depending on $s,\beta,\nu,d$. 
But when $|e+z|\geq 1$,  the integrand is clearly bounded above by such a constant; and when $|e+z|< 1$, then it is negative.  This therefore concludes the proof.



\end{document}